\documentclass[11pt]{article}
\usepackage[a4paper]{geometry}
\usepackage[linesnumbered,ruled,vlined]{algorithm2e}
\usepackage{amsfonts,amsmath,amsopn,amssymb}
\usepackage{braket}

\usepackage{enumitem}
\usepackage{graphicx}
\usepackage{epstopdf}
\usepackage{subcaption}
\usepackage{caption}
\usepackage[colorlinks,linkcolor=blue,anchorcolor=blue,citecolor=blue]{hyperref}
\usepackage{cleveref}
\usepackage{ntheorem}
\usepackage{url}
\usepackage{xcolor}
\usepackage{booktabs}
\usepackage{multirow}
\usepackage{diagbox}
\numberwithin{equation}{section}

\theoremheaderfont{\normalfont\bf}
\theorembodyfont{\normalfont\rm}

\crefname{algorithm}{Algorithm}{Algorithms}
\Crefname{algorithm}{Algorithm}{Algorithms}

\newtheorem{definition}{Definition}[section]
\newtheorem{remark}{Remark}[section]
\newtheorem{theorem}{Theorem}[section]

\newtheorem{proposition}{Proposition}[section]

\newenvironment{keywords}{{\noindent\it {\bf Key~words.}}\quad}{}
\newenvironment{proof}{{\noindent\it Proof.}\quad}{\hfill $\square$\par}

\DeclareMathOperator{\rank}{rank}

\SetKwInput{KwInput}{Input}         
\SetKwInput{KwOutput}{Output}       

\begin{document}
\title{QR Decomposition of Dual Matrices and its Application to Traveling Wave Identification in the Brain}
\author{Renjie Xu \thanks{Center for Intelligent Multidimensional Data Analysis, Hong Kong Science Park, Shatin, Hong Kong, China.  Email: \href{mailto:renjie@innocimda.com}{renjie@innocimda.com.} This author is supported by  the Hong Kong Innovation and Technology Commission (InnoHK Project CIMDA).}
	\and Tong Wei \thanks{Institute of Science and Technology for Brain-Inspired Intelligence, Fudan University, Shanghai, 200433, P.R. China. Email: \href{mailto:20110850008@fudan.edu.cn}{20110850008@fudan.edu.cn.}
    This author is partially supported by Science and Technology Commission of Shanghai Municipality (No. 23ZR1403000, 20JC1419500, 2018SHZDZX0).}
    \and  Yimin Wei \thanks{ School of Mathematical Sciences and Shanghai Key Laboratory of Contemporary Applied Mathematics, Fudan University, Shanghai, 200433, P.R. China. Email: \href{mailto:ymwei@fudan.edu.cn}{ymwei@fudan.edu.cn.} This author is supported by the National Natural Science Foundation of China under grant 12271108 and the Ministry of Science and Technology of China under grant G2023132005L.}
    \and Pengpeng Xie \thanks{Corresponding author (P. Xie).      School of Mathematical Sciences,  Ocean University of China, Qingdao 266100, P.R. China. \href{mailto:xie@ouc.edu.cn }{xie@ouc.edu.cn.}This author is supported by the National Natural Science Foundation of China under grant 12271108, Fundamental Research Funds for the Central Universities (No. 202264006), Laoshan
    	Laboratory (LSKJ202202302) and Qingdao Natural Science Foundation 23-2-1-158-zyyd-jch. .}
}
\date{}
\maketitle

\begin{abstract}
Matrix decompositions in dual number representations have played an important role in fields such as kinematics and computer graphics in recent years.
In this paper, we present a QR decomposition algorithm for dual number matrices, specifically geared towards its application in traveling wave identification, utilizing the concept of proper orthogonal decomposition.
When dealing with large-scale problems, we present explicit solutions for the QR, thin QR, and randomized QR decompositions of dual number matrices, along with their respective algorithms with column pivoting.
The QR decomposition of dual matrices is an accurate first-order perturbation, with the Q-factor satisfying rigorous perturbation bounds, leading to enhanced orthogonality.
In numerical experiments, we discuss the suitability of different QR algorithms when confronted with various large-scale dual matrices, providing their respective domains of applicability.
Subsequently, we employed the QR decomposition of dual matrices to compute the DMPGI, thereby attaining results of higher precision.
 Moreover, we apply the QR decomposition in the context of traveling wave identification, employing the notion of proper orthogonal decomposition to perform a validation analysis of large-scale functional magnetic resonance imaging (fMRI) data for brain functional circuits.
Our approach significantly improves the identification of two types of wave signals compared to previous research, providing empirical evidence for cognitive neuroscience theories.
\end{abstract}


\bigskip



\section{Introduction}\label{sec: introduction}
Time series analysis is a technique for uncovering hidden information and patterns within time-ordered data.
Principal Component Analysis (PCA) is employed to interpret the meaning of principal components through orthogonal transformations in time series processes \cite{yang2004pca}.
In statistics, only a small fraction of components contain a significant amount of initial information, and their importance decreases sequentially.
However, despite the orthogonality between these principal components, they may not necessarily be independent.

A natural idea is to utilize dual numbers to augment the original data with differentials related to time.
This concept has been applied in the context of traveling wave detection \cite{wei2024singular}.
However, this approach exhibits limitations in terms of orthogonality, particularly in the application of brainwave detection.
In this paper, we introduce the notion of appropriate orthogonal decomposition, considering the fact that the target time series also represents signal waves.
We propose the QR decomposition of dual matrices as an effective and efficient solution to enhance traveling wave detection.

 Clifford \cite{clifford1871preliminary} initially introduced dual numbers in 1873 as a hyper-complex number system, constitute an algebraic structure denoted as $a + b\epsilon$, where `$a$' and `$b$' are real numbers, and $\epsilon$ satisfies $\epsilon^{2} = 0$ with $\epsilon \neq 0$.
These dual numbers can be viewed as an infinitesimal extension of real numbers, representing either minute perturbations of the standard part or time derivatives of time series data.

Dual numbers find applications across diverse domains, including robotics \cite{sola2017quaternion}, computer graphics \cite{peon2019dual}, control theory \cite{fliess2013model}, optimization \cite{baydin2018automatic}, automatic differentiation \cite{fike2012automatic}, and differential equations \cite{fornberg1988generation}, owing to their distinctive properties.
Consequently, researchers are increasingly exploring the fundamental theoretical aspects of dual matrices, encompassing eigenvalues \cite{qi2021eigenvalues}, the singular value decomposition (SVD) \cite{gutin2022generalizations, wei2024singular}, the low-rank approximations \cite{qi2022low,xiong2022low,xu2024utv}, dual Moore-Penrose pseudo-inverse \cite{angeles2012dual, cui2023perturbations,pennestri2018moore,udwadia2021dual,udwadia2021does,udwadia2020all,wang2021characterizations}, dual rank decomposition \cite{wang2023dual}, and the QR decomposition \cite{pennestri2007linear}, among others.

The properties and algorithms for dual matrix decomposition have long been a focal point of research.
Previous studies primarily utilized the Gram-Schmidt orthogonalization process or redundant linear systems via Kronecker products to compute matrix decompositions.
For instance, Pennestr{\`\i} {\it et al.} \cite{pennestri2018moore} employed the Gram-Schmidt orthogonalization process to compute the QR decomposition of dual matrices.
Qi {\it et al.} \cite{qi2023eigenvalues} used redundant linear systems based on Kronecker products to solve the Jordan decomposition of dual complex matrices with a Jordan block standard part.
These approaches are not well-posed for solving large-scale problems.

Recently, Wei {\it et al.} \cite{wei2024singular} introduced a compact dual singular value decomposition (CDSVD) algorithm with explicit solutions.
By observing the solving process, we realize that the essence of computing these decompositions is solving a generalized Sylvester equation.
Building upon this idea, we provide the QR decomposition of dual matrices (DQR), the thin QR decomposition of dual matrices (tDQR), and the randomized QR decomposition of dual matrices (RDQR), all of which account for column-pivoting scenarios.
The QR decomposition of dual matrices also offers an effective approach for computing the dual Moore-Penrose generalized inverse (DMPGI), ensuring heightened precision in the resultant outcomes.
The DQR, tDQR, and RDQR proposed here are the proper orthogonal decomposition that we require in time series models.
In traveling wave detection, we can see that the DQR, tDQR, and RDQR algorithms perfectly fulfill this function.
Moreover, it exhibits stronger orthogonality in large-scale functional magnetic resonance imaging (fMRI) data on brain activity, enabling better identification of traveling wave signals in brain functional areas.
In large-scale problems, we can also observe that our proposed QR decomposition is faster than the CDSVD.

The rest of this paper is structured as follows.
In Section \ref{sec: Preliminaries}, we present some operational rules and notations for dual numbers and introduce the Sylvester equation.
In Section \ref{sec: QR Decomposition of Dual matrices}, we introduce the concepts of proper orthogonal decomposition for the dual matrix, leading to the development of the QR decomposition (DQR), thin QR decomposition (tDQR), and randomized QR decomposition (RDQR) for the dual matrix, all of which account for column-pivoting
scenario.
We demonstrate the conditions for the existence of these decompositions and provide algorithms for their computations.
In Section \ref{sec: Perturbation Analysi}, we provide a perturbation analysis perspective to demonstrate that the QR decomposition of dual matrices is an accurate first-order perturbation, satisfying the rigorous perturbation bounds for the Q-factor in the QR decomposition.
This confirms its improved orthogonality.
In Section \ref{sec: Numerical Exiperment}, we compare the computational times of different QR decompositions for large-scale dual matrices and find that their applicability aligns with theoretical expectations.
Subsequently, we employed the QR decomposition of dual matrices to compute the DMPGI, thereby attaining results of higher precision.
Moreover, we verify that the QR decomposition, based on the idea of proper orthogonal decomposition, can effectively identify traveling waves.
Additionally, we demonstrate its utility in accurately characterizing brain functional regions using functional magnetic resonance imaging (fMRI) data on wave patterns.

\section{Preliminaries}\label{sec: Preliminaries}
In this section, we briefly review some basic results for dual matrices and the Sylvester equation.
Furthermore, scalars, vectors, and matrices are denoted by lowercase ($a$), bold lowercase ($\bf{a}$),  and uppercase letters ($A$), respectively.
In some representations of matrix results, we will employ the notation used in MATLAB.

\subsection{Dual numbers}
In mathematics, dual numbers form an extension of real numbers by introducing a second component with the infinitesimal unit $\epsilon$ satisfying $\epsilon^{2} = 0$ and $\epsilon\neq 0$.
This component plays a crucial role in capturing infinitesimal changes and uncertainties in various mathematical operations.

We denote $\mathbb{DR}$ as the set of dual real numbers. Let $d = d_{s} +  d_{i}\epsilon$ be a dual real number where $d_{s}\in \mathbb{R}$ and  $d_{i}\in \mathbb{R}$.
In this paper, $d_{s}$ and $d_{i}$ are represented as the standard part and the infinitesimal part of $d$, respectively.
If $d_{s} \neq 0$, we say that $d$ is appreciable. Otherwise, $d$ is infinitesimal.
For two real matrices $A_{s},A_{i}\in \mathbb{R}^{m \times n}$, $A = A_{s} + A_{i} \epsilon \in \mathbb{DR}^{m \times n}$  represents a dual real  matrix, where $\mathbb{DR}^{m \times n}$ is the set of dual real matrices.
The following proposition shows the preliminary properties of dual real matrices.

\begin{proposition}\label{prop: dual properties}
    \cite{qi2021eigenvalues} Suppose that $A = (a_{ij}) = A_{s}+A_{i} \epsilon \in \mathbb{DR}^{m\times n}$, $B = B_{s} +B_{i}\epsilon \in \mathbb{DR}^{n\times p}$ and $C =C_{s} +C_{i}\epsilon \in \mathbb{DR}^{n \times n}$. Then
    \begin{enumerate}
        \item The transpose of $A$ is $A^{\top} = (a_{ji}) = A_{s}^{\top}+A_{i}^{\top}\epsilon$.
        Moreover, $(AB)^{\top} = B^{\top} A^{\top}.$
        \item The dual matrix $C$ is diagonal, if the matrices $C_{s}$ and $C_{i}$ are both diagonal.
        \item The dual matrix $C$ is orthogonal, if $C^{\top}C = I_{n}$, where $I_{n} $ is an $n \times n$ identity matrix. Furthermore, $CC^{\top} = I_{n}$, confirms the orthogonal property of the dual matrix $C$.
        \item The dual matrix $C$ is invertible (nonsingular), if $CD = DC =I_{n}$ for some $D \in \mathbb{DR}^{n \times n}$. Furthermore, $D$ is unique and denoted by $C^{-1}$, where $C^{-1} = C_{s}^{-1}-C_{s}^{-1}C_{i}C_{s}^{-1} \epsilon$.
        \item When $n\geq p$, the dual matrix $B$ is defined to have orthogonal columns if $B^{\top}B = I_{n}$. In particular, $B_{s}$ is a real matrix with orthogonal columns, and $B_{s}^{\top}B_{i} + B_{i}^{\top}B_{s} = O_{p}$, where $O_{p}$ denotes a $p \times p$ zero matrix.
    \end{enumerate}
\end{proposition}



\subsection{Sylvester equation}\label{sec: Sylvester Equation}
The Sylvester equation, named in honor of the 19th-century mathematician James Joseph Sylvester, represents a prominent matrix equation within the domain of linear algebra.
It finds extensive utility across diverse realms of scientific and engineering disciplines, encompassing control theory, signal processing, system stability analysis, and numerical analysis.
The Sylvester equation emerges prominently in inquiries linked to time-invariant linear systems, with its solutions serving as indispensable tools for elucidating the characteristics and attributes of these intricate systems.
The equation is of the form,
\begin{equation}\label{equ: Sylvester equation }
    AX + XB = C,
\end{equation}
where $A, B$, and $C$ are given matrices of appropriate dimensions, and $X$ is the unknown matrix that we seek to find.
 There exists a unique solution $X$ exactly when the spectra of $A$ and $-B$ are disjoint.
Here, we first introduce the concept of the Moore-Penrose pseudoinverse.

\begin{definition}\label{def: Moore Penrose Inverse}
    \cite{golub2013matrix} Given $A\in \mathbb{R}^{m\times n}$, its Moore-Penrose pseudoinverse, denoted as $A^{\dagger}$ which is a matrix that satisfies the following four properties,
    \begin{equation}
        AA^{\dagger}A = A, \quad A^{\dagger}AA^{\dagger} = A^{\dagger}, \quad (AA^{\dagger})^{\top} = AA^{\dagger}, \quad (A^{\dagger}A)^{\top} = A^{\dagger}A.
    \end{equation}
\end{definition}
Then the solution form of the generalized Sylvester equation can be given by the following theorem.
\begin{theorem}\label{the: solution of Sylvester equation}
    \cite{Baksalary1979TheME} Given $A\in \mathbb{R}^{m \times k}, B\in \mathbb{R}^{l \times n}$, and $C\in\mathbb{R}^{m \times n}$. The equation
    \begin{equation}\label{equ: Ax-YB =C}
        AX-YB = C,
    \end{equation}
    has a solution $X\in \mathbb{R}^{k \times n},Y \in \mathbb{R}^{m \times l}$ if and only if
    \begin{equation}
        (I_{m}-AA^{\dagger})C(I_{n}-B^{\dagger}B) = O_{m\times n}.
    \end{equation}
    If this is the case, then the general solution of (\ref{equ: Ax-YB =C}) has the form
    \begin{equation}
        \begin{cases}
        X = A^{\dagger}C+A^{\dagger}ZB+(I_{k}-A^{\dagger}A)W,\\
        Y = -(I_{m}-AA^{\dagger})CB^{\dagger}+Z-(I_{m}-AA^{\dagger})ZBB^{\dagger},
        \end{cases}
    \end{equation}
    with $W\in \mathbb{R}^{k\times n}$, $Z\in\mathbb{R}^{m\times l}$  being arbitrary and $A^\dagger$ representing the Moore-Penrose pseudoinverse of $A$.
\end{theorem}
Theorem \ref{the: solution of Sylvester equation} constitutes a pivotal conclusion, holding significant relevance in the subsequent decomposition of dual matrices.

\section{QR Decomposition of Dual Matrices}\label{sec: QR Decomposition of Dual matrices}

In prior research \cite{wei2024singular}, the compact SVD decomposition of dual matrices is employed to perform a proper orthogonal decomposition of dual matrices for traveling wave identification.
This prompted our exploration of whether QR decomposition could be leveraged for proper orthogonal decomposition in the context of traveling wave identification.
In this section, we initially establish the existence of QR decomposition for dual matrices and propose an algorithm.
Subsequently, we present an algorithm for QR decomposition with column pivoting in the standard part of dual matrices.
We then extend our approach by introducing the concept of thin QR decomposition \cite{golub2013matrix} and presenting algorithms for both thin QR decomposition of dual matrices and thin QR decomposition with column pivoting in the standard part.
Furthermore, we harness the idea of randomized QR decomposition with column pivoting \cite{duersch2017randomized} to formulate an algorithm for randomized QR decomposition with column pivoting in the standard part of dual matrices.
The latter two categories of algorithms are specifically designed for addressing large-scale data scenarios where $m \gg n$, enables a proper orthogonal decomposition.

\subsection{QR Decomposition of Dual Matrices (DQR)}\label{sec: DQR}
\begin{theorem}\label{the: QR decomposition fo dual matrices}
    (DQR Decomposition) Given $A = A_{s}+A_{i}\epsilon \in \mathbb{DR}^{m \times n}$. Assume that $A_{s} = Q_{s}R_{s}$ is the QR decomposition of $A_{s}$, where $Q_{s} \in \mathbb{R}^{m \times m}$ is an orthogonal matrix and $R\in \mathbb{R}^{m \times n}$ is an upper triangular matrix. Then the QR decomposition of the dual matrix $A = QR$ where $Q = Q_{s}+Q_{i}\epsilon \in \mathbb{DR}^{m \times m}$ is an orthogonal dual matrix and $R = R_{s}+R_{i}\epsilon \in \mathbb{R}^{m \times n}$ is an upper triangular dual matrix exists. Additionally, the expressions of $Q_{i}$, $R_{i}$, the infinitesimal parts of $Q,R$, are
    \begin{equation}\label{equ: Qi=QsP}
        \begin{cases}
            Q_{i} & = Q_{s}P\\
            R_{i} & = Q_{s}^{\top}A_{i} -PR_{s},
        \end{cases}
    \end{equation}
    where $P\in\mathbb{R}^{m \times m}$ is a skew-symmetric matrix.

    Assuming $\rank(A_{s}) = K \leq \min(m,n)$ and $R_{s}
    = (r_{s_{ij}})\in \mathbb{R}^{m\times n}$,  we can derive the lower triangular portion's expression for $P$ as
    \begin{equation}\label{equ: p_{1} &= b_{1}/r_{s_{11}}}
        \begin{cases}
            p_{1}(2:m) &= b_{1}(2:m)/r_{s_{11}}, \\
            p_{2}(3:m) &= (b_{2}(3:m)-r_{s_{12}}p_{1}(3:m))/r_{s_{22}}, \\
            & \ldots \\
            p_{K}(K+1:m) &= (b_{K}(K+1:m)-\sum\limits_{t = 1}^{K-1}r_{s_{tK}}p_{t}
            (K+1:m))/r_{s_{KK}},
        \end{cases}
    \end{equation}
    where $Q_{s}^{\top}A_{i} = B = [b_{1},b_{2},\ldots,b_{n}]$ and the remaining lower triangular portion is entirely composed of zeros.
\end{theorem}
\begin{proof}
    Assume the QR decomposition of $A$ exists, then we have that
    \begin{equation}\label{equ: As=QsRs Ai = QsRi + QiRs}
        \begin{cases}
            A_{s} = Q_{s}R_{s}\\
            A_{i} = Q_{s}R_{i}+Q_{i}R_{s}
        \end{cases}
    \end{equation}
    by simple computation.

    By treating $R_{i}$ as $X$ and $Q_{i}$ as $Y$, the given equation $A_{i} = Q_{s}R_{i}+Q_{i}R_{s}$ in (\ref{equ: As=QsRs Ai = QsRi + QiRs}) transforms into a Sylvester equation involving $Q_{i}$ and $R_{i}$. Leveraging  Theorem \ref{the: solution of Sylvester equation}, if and only if
    \begin{equation}\label{equ: (I-QQT)A(I-RR)}
        (I_{m}-Q_{s}Q_{s}^{\dagger})A_{i}(I_{m}-R_{s}^{\dagger}R_{s}) = O_{m\times n},
    \end{equation}
    which holds true identically by the orthogonal nature of $Q_{s}$ and the fact $Q_{s}^{\dagger} = Q_{s}^{\top}$, one can derive the solution
    \begin{equation}\label{equ: RiQi after Sylvester}
        \begin{cases}
            R_{i} &= Q_{s}^{\dagger}A_{i}+Q_{s}^{\dagger}Z(-R_{s})+(I_{m}-Q_{s}^{\dagger}Q_{s})W \\
            Q_{i} &= -(I_{m}-Q_{s}Q_{s}^{\dagger})A_{i}(-R_{s})^{\dagger}+Z-(I_{m}-Q_{s}Q_{s}^{\dagger})Z(-R_{s})(-R_{s})^{\dagger}
        \end{cases}
    \end{equation}
    where $Z \in \mathbb{R}^{m \times m}$ and $W \in \mathbb{R}^{m \times n}$ are
    arbitrary.

    Upon examining the first expression of $R_{i}$ in (\ref{equ: RiQi after Sylvester}), one can deduce, based on the established fact that $Q_{s}^{\dagger} = Q_{s}^{\top}$,
    \begin{equation}\label{equ: Ri expression}
        \begin{aligned}
            R_{i} & = Q_{s}^{\dagger}A_{i}+Q_{s}^{\dagger}Z(-R_{s})+(I_{m}-Q_{s}^{\dagger}Q_{s})W \\
            & = Q_{s}^{\top}A_{i}+Q_{s}^{\top}Z(-R_{s})+(I_{m}-Q_{s}^{\top}Q_{s})W \\
            & = Q_{s}^{\top}A_{i}-Q_{s}^{\top}ZR_{s}.
        \end{aligned}
    \end{equation}
    Likewise, drawing on the foundational premise of $Q_{s}^{\dagger} = Q_{s}^{\top}$
    and $I_{m}-Q_{s}Q_{s}^{\top} = O_{m }$, we can derive the second expression of $Q_{i}$ in (\ref{equ: RiQi after Sylvester}) into
    \begin{equation}\label{equ: Qi expression}
        \begin{aligned}
            Q_{i} & = -(I_{m}-Q_{s}Q_{s}^{\dagger})A_{i}(-R_{s})^{\dagger}+Z-(I_{m}-Q_{s}Q_{s}^{\dagger})Z(-R_{s})(-R_{s})^{\dagger}\\
            & = (I_{m}-Q_{s}Q_{s}^{\top})A_{i}R_{s}^{\dagger}+Z-(I_{m}-Q_{s}Q_{s}^{\top})ZR_{s}R_{s}^{\dagger} \\
            & = Z.
        \end{aligned}
    \end{equation}
    Subsequently, the determination of arbitrary values for $Z$ becomes imperative. Notably, the orthogonal nature of the dual matrix $Q$ and the upper triangular configuration of the dual matrix $R$ can be employed as constraining equations to ascertain the values of $Z$. Primarily, our scrutiny revolves around the orthogonal property of the dual matrix $Q$. This leads us to consider the following
    \begin{equation}\label{equ: Qi*Qs+Qs*Qi=0}
        Q_{i}^{\top}Q_{s}+Q_{s}^{\top}Q_{i} = O_{m}.
    \end{equation}
    Substituting (\ref{equ: Qi expression}) into the term $Q_{s}^{\top}Q_{i}$ in (\ref{equ: Qi*Qs+Qs*Qi=0}) yields
    \begin{equation}\label{equ: QsTQi = QsTZ}
            Q_{s}^{\top}Q_{i}  = Q_{s}^{\top}Z.
    \end{equation}
    Hence, (\ref{equ: Qi*Qs+Qs*Qi=0}) can be transformed into
    \begin{equation}\label{equ: ZTQs+QsTZ = O}
        Z^{\top}Q_{s} + Q_{s}^{\top}Z = O_{m}.
    \end{equation}
    Here, owing to the orthogonal and invertible nature of matrix $Q_{s}$, we can denote $P = Q_{s}^{\top}Z \in \mathbb{R}^{m \times m}$ (thus equivalently, $Z = Q_{s}P$, then the investigation into the values of $Z$ can be transmuted into an exploration of the values pertaining to $P$). According to (\ref{equ: QsTQi = QsTZ}), it follows that $P$ is a skew-symmetric matrix. Through the skew-symmetric matrix $P = Q_{s}^{\top}Z$ , we can attain expressions for $Q_{i}$ and $R_{i}$,
    \begin{equation}
        \begin{cases}
            Q_{i} & = Q_{s}P\\
            R_{i} & = Q_{s}^{\top}A_{i} -PR_{s}.
        \end{cases}
    \end{equation}

    Now, we must revisit the upper triangular property of the dual matrix $R$ to determine $P$. Upon substituting the expression $P = Q_{s}^{\top}Z$ into (\ref{equ: Ri expression}), we obtain
    \begin{equation}\label{equ: PRs=QsTAi-Ri}
        PR_{s} = Q_{s}^{\top}A_{i}-R_{i}.
    \end{equation}
    We denote $B = Q_{s}^{\top}A_{i} \in\mathbb{R}^{m \times n}$ and $B = [b_{1},b_{2},\ldots,b_{n}], R_{i} = [r_{i1},r_{i2},\ldots,r_{in}], R_{s} = [r_{s1},r_{s2},\ldots,r_{sn}]$ where $b_{k},r_{ik},r_{sk}\in \mathbb{R}^{m},k = 1,2,\ldots,n $. By partitioning the matrices column-wise, (\ref{equ: PRs=QsTAi-Ri}) is transformed into the following system of linear equations,
    \begin{equation}\label{equ: Prs1 = b1-ri1}
        \begin{cases}
            Pr_{s1} &= b_{1}-r_{i1}, \\
            Pr_{s2} &= b_{2}-r_{i2}, \\
            & \ldots \\
            Pr_{sk} &= b_{k}-r_{ik},\\
            & \ldots
        \end{cases}
    \end{equation}
    Exploiting the upper triangular nature of matrix $R_{s}$, along with  $R_{s}
     = (r_{s_{ij}})\in \mathbb{R}^{m\times n}$ and $P = [p_{1},p_{2},\ldots,p_{n}]$ where $p_{k}\in \mathbb{R}^{m},k = 1,2,\ldots, m$, the system (\ref{equ: Prs1 = b1-ri1}) can be transmuted into
    \begin{equation}\label{equ: rs11p1 = b1-ri1}
        \begin{cases}
            r_{s_{11}}p_{1} &= b_{1}-r_{i1}, \\
            r_{s_{12}}p_{1}+r_{s_{22}}p_{2} &= b_{2}-r_{i2}, \\
            & \ldots \\
            r_{s_{1k}}p_{1}+r_{s_{2k}}p_{2}+\ldots+r_{s_{kk}}p_{k} &= b_{k}-r_{ik},\\
            & \ldots
        \end{cases}
    \end{equation}
    Recognizing the upper triangular structure of $R_{i}$, it follows that for $r_{ik},k = 1,2,\ldots, n$, the lower $n-k$ entries $r_{ik}(k+1:m)$ are all zeros. Consequently, $p_{k}$ also requires computation solely for the lower $n-k$ entries $p_{k}(k+1:m)$, thereby ensuring the maintenance of $P$'s
    skew-symmetric property.

     Given $\rank(R_{s}) = \rank(A_{s}) =  K \leq \min(m,n)$, it follows that all the
     elements $r_{s_{ij}} = 0, i>K $. Hence, the system of $K$ equations has the solutions
    \begin{equation}
        \begin{cases}
            p_{1}(2:m) &= b_{1}(2:m)/r_{s_{11}}, \\
            p_{2}(3:m) &= (b_{2}(3:m)-r_{s_{12}}p_{1}(3:m))/r_{s_{22}}, \\
            & \ldots \\
            p_{K}(K+1:m) &= (b_{K}(K+1:m)-\sum\limits_{t = 1}^{K-1}r_{s_{tK}}p_{t}(K+1:m))/r_{s_{KK}}.
        \end{cases}
    \end{equation}
     Thus, we complete the proof.
\end{proof}

Theorem \ref{the: QR decomposition fo dual matrices} furnishes the proof of existence for the QR decomposition of the dual matrices, accompanied by a delineation of the construction process within the proof. In Algorithm \ref{alg: QR Decomposition of Dual Matrix}, we furnish pseudocode for the QR decomposition of the dual matrices.

\begin{algorithm}[htb]
\DontPrintSemicolon
    \KwInput{$A = A_{s}+A_{i}\epsilon \in \mathbb{DR}^{m \times n}$.}

     Compute the QR decomposition of $A_{s} = Q_{s}R_{s}(Q_{s}\in \mathbb{R}^{m \times m},R_{s}\in \mathbb{R}^{m \times n})$.

    Compute $K = \rank(R_{s})$.

    Set $P = zeros(m,m)$ and $B = Q_{s}^{\top}A_{i}$.

    Set $i = 1$.

    $P(2:m,1) = B(2:m,1)/R_{s}(1,1)$.

    \While{$i < K$}{

    $i = i+1$.

    $P(i+1:m,i) = (B(i+1:m,i)-\sum\limits_{t=1}^{i-1}R_{s}(t,i)P(i+1:m,t)/R_{s}(i,i).$

    $P(i,i+1:m) = -P(i+1:m,i)^{\top}$.

    }

    Set $Q_{i} = Q_{s}P$.

    Set $R_{i}  = B -PR_{s}$.

    \KwOutput{$Q = Q_{s}+Q_{i}\epsilon \in \mathbb{DR}^{m \times m}$ is an orthogonal dual matrix, and $R = R_{s}+R_{i}\epsilon \in \mathbb{DR}^{m \times n}$ is an upper triangular dual matrix.}
    \caption{QR Decomposition of Dual Matrices (DQR)}
    \label{alg: QR Decomposition of Dual Matrix}
\end{algorithm}

In practical traveling wave detection \cite{wei2024singular}, it becomes necessary to perform column pivoting in the standard components of the QR decomposition. Thus, suppose that the QR decomposition with column pivoting in the standard part of $A = A_{s}+A_{i}\epsilon \in \mathbb{DR}^{m \times n}$ exists, then we have
\begin{equation}\label{equ: AsPs=QsRs}
    \begin{cases}
        A_{s}P_{s} &= Q_{s}R_{s}\\
        A_{i}P_{s} &= Q_{s}R_{i}+Q_{i}R_{s}.
    \end{cases}
\end{equation}
Building upon (\ref{equ: AsPs=QsRs}), we can employ a similar approach as in Theorem \ref{the: QR decomposition fo dual matrices} to complete the proof. We omit the details here. Algorithm \ref{alg: QR decomposition with column pivoting in the standard part of Dual Matrix} presents the pseudocode for the QR decomposition with column pivoting in the standard part of the dual matrix (DQRCP).

\begin{algorithm}[htb]
\DontPrintSemicolon
    \KwInput{$A = A_{s}+A_{i}\epsilon \in \mathbb{DR}^{m \times n}$.}

     Compute the QR decomposition with column pivoting of $A_{s}P_{s} = Q_{s}R_{s}(Q_{s}\in \mathbb{R}^{m \times m},R_{s}\in \mathbb{R}^{m \times n}, P_{s} \in \mathbb{R}^{n \times n})$.

    Set $P = zeros(m,m)$ and $B = Q_{s}^{\top}A_{i}P_{s}$.

    Set $i = 1$.

    $P(2:m,1) = B(2:m,1)/R_{s}(1,1)$.

    \While{$i < K$}{

    $i = i+1$.

    $P(i+1:m,i) = (B(i+1:m,i)-\sum\limits_{t=1}^{i-1}R_{s}(t,i)P(i+1:m,t)/R_{s}(i,i)$.

    $P(i,i+1:m) = -P(i+1:m,i)^{\top}$.

    }

    Set $Q_{i} = Q_{s}P$.

    Set $R_{i}  = B -PR_{s}$.

    \KwOutput{$Q = Q_{s}+Q_{i}\epsilon \in \mathbb{DR}^{m \times m}$ is an orthogonal dual matrix, $R = R_{s}+R_{i}\epsilon \in \mathbb{DR}^{m \times n}$ is an upper triangular dual matrix, and $P_{s} \in \mathbb{R}^{n \times n}$ is a permutation matrix.}
    \caption{QR Decomposition with Column Pivoting in the standard part of Dual Matrices (DQRCP)}
    \label{alg: QR decomposition with column pivoting in the standard part of Dual Matrix}
\end{algorithm}

\subsection{Thin QR Decomposition of Dual Matrices (tDQR)}\label{sec: tDQR}

In Algorithm \ref{alg: QR Decomposition of Dual Matrix}, it is evident that while computing the QR decomposition of the dual matrix $A\in \mathbb{R}^{m \times n}$, in order to construct the skew-symmetric matrix $P$, a space of dimension $m^{2}$ is required. In the context of large-scale numerical computations, a large value of $m$ can have a detrimental impact on our calculations. Therefore, this necessitates our contemplation of thin QR decomposition of dual matrices. The concept of thin QR decomposition was initially introduced by Golub and Van Loan \cite{golub2013matrix}. For a matrix $A\in \mathbb{R}^{m \times n} (m \geq n)$ with full column rank, the thin QR decomposition $A = QR$ entails the unique existence of a matrix $Q\in \mathbb{R}^{m \times n}$ with columns orthogonality and an upper triangular matrix $R\in \mathbb{R}^{n \times n}$ with positive diagonal entries. Subsequently, we proceed to present the thin QR decomposition of the dual matrices (tDQR decomposition) based on this notion.

\begin{theorem}\label{the: thin QR decomposition fo dual matrices}
    (tDQR Decomposition) Given $A = A_{s}+A_{i}\epsilon \in \mathbb{DR}^{m \times n}(m \geq n)$ where $A_{s} \in \mathbb{R}^{m \times n}$ is full column rank. Assume that $A_{s} = Q_{s}R_{s}$ is the thin QR decomposition of $A_{s}$, where $Q_{s} \in \mathbb{R}^{m \times n}$ is a matrix with orthogonal columns and $R\in \mathbb{R}^{n \times n}$ is an upper triangular matrix with positive diagonal entries. Then the thin QR decomposition of the dual matrix $A = QR$ where $Q = Q_{s}+Q_{i}\epsilon \in \mathbb{DR}^{m \times n}$ is a dual matrix with orthogonal columns and $R = R_{s}+R_{i}\epsilon \in \mathbb{R}^{n \times n}$ is an upper triangular dual matrix exists. Additionally, the expressions of $Q_{i}$, $R_{i}$, the infinitesimal parts of $Q,R$, are
    \begin{equation}\label{equ: thin Qi Ri expression}
        \begin{cases}
            Q_{i} & = (I_{m}-Q_{s}Q_{s}^{\top})A_{i}R_{s}^{-1} + Q_{s}P \\
            R_{i}& = Q_{s}^{\top}A_{i}-PR_{s}
        \end{cases}
    \end{equation}
    where $P \in \mathbb{R}^{n \times n}$ is a skew-symmetric matrix.

    We can derive the lower triangular portion's expression for $P$ as
    \begin{equation}
        \begin{cases}
            p_{1}(2:n) &= b_{1}(2:n)/r_{s_{11}}, \\
            p_{2}(3:n) &= (b_{2}(3:n)-r_{s_{12}}p_{1}(3:n))/r_{s_{22}}, \\
            & \ldots \\
            p_{n-1}(n) &= (b_{n-1}(n)-\sum\limits_{t = 1}^{n-2}r_{s_{t(n-1)}}p_{t}(n))/r_{s_{(n-1)(n-1)}},
        \end{cases}
    \end{equation}
    where $Q_{s}^{\top}A_{i} = B = [b_{1},b_{2},\ldots,b_{n}]$.
\end{theorem}
The proof can be found in the Appendix \ref{sec: appendix}.

Theorem \ref{the: thin QR decomposition fo dual matrices} provides the thin QR decomposition of dual matrices (tDQR), offering an efficient alternative for the QR decomposition of dual matrices (DQR) when dealing with large-scale problems where $m \gg n$. Algorithm \ref{alg: thin QR Decomposition of Dual Matrix} furnishes pseudocode for the tDQR.
\begin{algorithm}[htb]
\DontPrintSemicolon
    \KwInput{$A = A_{s}+A_{i}\epsilon \in \mathbb{DR}^{m \times n}$ where $A_{s}$ is of full column rank.}

     Compute the thin QR decomposition of $A_{s} = Q_{s}R_{s}(Q_{s}\in \mathbb{R}^{m \times n},R_{s}\in \mathbb{R}^{n \times n})$.

    Set $P = zeros(n,n)$ and $B = Q_{s}^{\top}A_{i}$.

    Set $i = 1$.

    $P(2:n,1) = B(2:n,1)/R_{s}(1,1)$.

    \While{$i < n-1 $}{

    $i = i+1$.

    $P(i+1:n,i) = (B(i+1:n,i)-\sum\limits_{t=1}^{i-1}R_{s}(t,i)P(i+1:n,t)/R_{s}(i,i)$.

    $P(i,i+1:n) = -P(i+1:n,i)^{\top}$.

    }

    Set $Q_{i} = (I_{m}-Q_{s}Q_{s}^{\top})A_{i}R_{s}^{-1} + Q_{s}P$.

    Set $R_{i}  = B -PR_{s}$.

    \KwOutput{$Q = Q_{s}+Q_{i}\epsilon \in \mathbb{DR}^{m \times n}$ is a dual matrix with orthogonal columns, and $R = R_{s}+R_{i}\epsilon \in \mathbb{DR}^{n \times n}$ is an upper triangular dual matrix.}
    \caption{Thin QR Decomposition of Dual Matrices (tDQR)}
    \label{alg: thin QR Decomposition of Dual Matrix}
\end{algorithm}

Similar to Algorithm \ref{alg: QR decomposition with column pivoting in the standard part of Dual Matrix}, we can derive an algorithm for the thin QR decomposition with column pivoting in the standard part of the dual matrix based on analogous reasoning. Algorithm \ref{alg: thin QR decomposition with column pivoting in the standard part of Dual Matrix} presents the pseudocode for the thin QR decomposition with column pivoting in the standard part of the dual matrix (tDQRCP).

\begin{algorithm}[htb]
\DontPrintSemicolon
    \KwInput{$A = A_{s}+A_{i}\epsilon \in \mathbb{DR}^{m \times n}$, where $A_{s}$ is of full column rank.}

     Compute the thin QR decomposition with column pivoting of $A_{s}P_{s} = Q_{s}R_{s}(Q_{s}\in \mathbb{R}^{m \times n},R_{s}\in \mathbb{R}^{n \times n}, P_{s} \in \mathbb{R}^{n \times n})$.

   Set $P = zeros(n,n)$ and $B = Q_{s}^{\top}A_{i}P_{s}$.

    Set $i = 1$.

    $P(2:n,1) = B(2:n,1)/R_{s}(1,1)$.

    \While{$i < n-1 $}{

    $i = i+1$.

    $P(i+1:n,i) = (B(i+1:n,i)-\sum\limits_{t=1}^{i-1}R_{s}(t,i)P(i+1:n,t)/R_{s}(i,i)$.

    $P(i,i+1:n) = -P(i+1:n,i)^{\top}$.

    }

    Set $Q_{i} = (I_{m}-Q_{s}Q_{s}^{\top})A_{i}P_{s}R_{s}^{-1} + Q_{s}P$.

    Set $R_{i}  = B -PR_{s}$.

    \KwOutput{$Q = Q_{s}+Q_{i}\epsilon \in \mathbb{DR}^{m \times m}$ is a dual matrix with orthogonal columns, $R = R_{s}+R_{i}\epsilon \in \mathbb{DR}^{m \times n}$ is an upper triangular dual matrix, and $P_{s} \in \mathbb{R}^{n \times n}$ is a permutation matrix.}
    \caption{Thin QR Decomposition with Column Pivoting in the standard part of Dual Matrices (tDQRCP)}
    \label{alg: thin QR decomposition with column pivoting in the standard part of Dual Matrix}
\end{algorithm}

\subsection{Randomized QR Decomposition with Column Pivoting of Dual Matrices (RDQRCP)}\label{sec: Randomized QR Decomposition with Column Pivoting of Dual Matrices}

In practical applications utilizing QR decomposition with column pivoting, it is often sufficient to compute only the first $k$ rows of $R$. Given an $m \times  n$ matrix $A$, we seek a truncated QR decomposition for $k \leq \min(m,n)$
\begin{equation}\label{equ: AP approx QR}
    AP \approx QR,
\end{equation}
where $Q \in \mathbb{R}^{m \times k}$ is
orthonormal, $P\in \mathbb{R}^{n \times n}$ is a permutation, and $R \in \mathbb{R}^{k \times n}$ is upper triangular. Addressing this concern, Duersch and Gu devised a randomized algorithm for QR decomposition with column pivoting \cite{duersch2017randomized}. This truncated QR decomposition employing randomized algorithms proves to be highly effective for numerous low-rank matrix approximation techniques. According to the results of the Johnson-Lindenstrauss lemma \cite{wb1982extensions}, randomized sampling algorithms maintain a high probability guarantee of approximation error while simultaneously reducing communication complexity through dimensionality reduction. The Johnson-Lindenstrauss lemma implies that let $a_{j}$ represent the $j$-th column ($j = 1,2,\ldots,n$) of $A$.  We construct the sample columns $b_{j} = \Omega a_{j}$ by a Gaussian Independent Identically Distributed (GIID) compression matrix $\Omega \in \mathbb{R}^{l\times m}$. The 2-norm expectation values and variance of $b_{j}$ are,
\begin{equation}\label{equ: Ebj=laj}
    \mathbb{E}(\|b_{j}\|_{2}^{2}) = l\|a_{j}\|_{2}^{2} \text{ and } \mathbb{V}(\|b_{j}\|_{2}^{2}) = 2l\|a_{j}\|_{2}^{4}
\end{equation}
Moreover, the probability of successfully detecting all column norms as well as all distances between columns within a relative error $\tau$ for $0 < \tau < \frac{1}{2}$ is bounded by
\begin{equation}\label{equ: Pbj-bi/aj-ai}
    P\left(\left| \frac{\|b_{j}-b_{i}\|_{2}^{2}}{l\|a_{j}-a_{i}\|_{2}^{2}} -1 \right|\leq \tau\right)\geq 1-2e^{\frac{-l\tau^{2}}{4}(1-\tau)}.
\end{equation}
We can select the column with the largest approximate norm by using the sample matrix $B = \Omega A$. The remaining columns can be processed as represented in Algorithm \ref{alg: Single-Sample Randomized QRCP}, with detailed proofs available in the work by Duersch and Gu \cite{duersch2017randomized}.

\begin{algorithm}[htb]
\DontPrintSemicolon
    \KwInput{$A  \in \mathbb{R}^{m \times n}$, $k \ll \min(m,n)$ the desired approximation rank.}

     Set sample rank $l = k + p$ needed for acceptable sample error.

     Generate random $l \times m$ GIID compression matrix $\Omega$ and form the sample $B=\Omega A$.

     Get the QR decomposition with column pivoting of the sample, $BP = Q_{b}R_{b}(Q_{b}\in \mathbb{R}^{l \times l}, R_{b}\in \mathbb{R}^{l \times n}, P\in \mathbb{R}^{n \times n})$.

    Apply permutation $A^{(1)} = AP $.

    Compute the thin QR decomposition of $A^{(1)}(:,1:k) = QR_{11} (Q\in \mathbb{R}^{m \times k}, R_{11}\in \mathbb{R}^{k \times k})$.

    Finish $k$ rows of $R$ in remaining columns, $R_{12} = Q^{\top}(:,1:k)A^{(1)}(:,k+1:end)$.

    \KwOutput{$Q\in \mathbb{R}^{m\times k}$ is a matrix with orthogonal columns, $R \in \mathbb{R}^{k \times n}$ is a truncated upper trapezoidal matrix, and $P\in \mathbb{R}^{n \times n}$ is a permutation matrix.}
    \caption{Single-Sample Randomized QRCP (RQRCP)}
    \label{alg: Single-Sample Randomized QRCP}
\end{algorithm}
Building upon the randomized QR decomposition with column pivoting for (\ref{equ: AP approx QR}) as illustrated in Algorithm \ref{alg: Single-Sample Randomized QRCP}, we present a randomized QR decomposition with column pivoting in the standard part of dual matrices.
\begin{theorem}\label{the: RDQRCP}
    (RDQRCP Decomposition) Given $A = A_{s}+A_{i}\epsilon \in \mathbb{DR}^{m \times n}(m \geq n)$. Assume that $A_{s}P_{s} = Q_{s}R_{s}$ is the randomized QR decomposition with column pivoting of $A_{s}$, where $k \leq \min(m,n)$, $Q_{s} \in \mathbb{R}^{m \times k}$ is a matrix with orthogonal columns, $R\in \mathbb{R}^{k \times n}$ is a truncated upper trapezoidal matrix and $P_{s}\in \mathbb{R}^{n \times n}$ is a permutation matrix and $\rank(A_{s})\geq k$. Then the randomized QR decomposition with column pivoting in the standard part of the dual matrix $AP_{s} = QR$ where $Q = Q_{s}+Q_{i}\epsilon \in \mathbb{DR}^{m \times k}$ is a dual matrix with orthogonal columns and $R = R_{s}+R_{i}\epsilon \in \mathbb{R}^{k \times n}$ is a truncated upper trapezoidal dual matrix exists if and only if
    \begin{equation}\label{equ: (I-QsQsT)AiP_s(I-RsRs)}
        (I_{m}-Q_{s}Q_{s}^{\top})A_{i}P_{s}(I_{n}-R_{s}^{\dagger}R_{s}) = O_{m\times n}.
    \end{equation}
    Additionally, the expressions of $Q_{i}$, $R_{i}$, the infinitesimal parts of $Q,R$, are
    \begin{equation}\label{equ: rand Qi Ri expression}
        \begin{cases}
            Q_{i} & = (I_{m}-Q_{s}Q_{s}^{\top})A_{i}P_{s}R_{s}^{\dagger} + Q_{s}P \\
            R_{i}& = Q_{s}^{\top}A_{i}P_{s}-PR_{s},
        \end{cases}
    \end{equation}
    where $P\in \mathbb{R}^{k \times k}$ is a skew-symmetric matrix.

    We can derive the lower triangular portion's expression for $P$ as
    \begin{equation}
        \begin{cases}
            p_{1}(2:k) &= b_{1}(2:k)/r_{s_{11}}, \\
            p_{2}(3:k) &= (b_{2}(3:k)-r_{s_{12}}p_{1}(3:k))/r_{s_{22}}, \\
            & \ldots \\
            p_{k-1}(k) &= (b_{k-1}(k)-\sum\limits_{t = 1}^{k-2}r_{s_{t(k-1)}}p_{t}(k))/r_{s_{(k-1)
            	(k -1)}},
        \end{cases}
    \end{equation}
    where $Q_{s}^{\top}A_{i}P_{s} = B = [b_{1},b_{2},\ldots,b_{k},\ldots,b_{n}]$.
\end{theorem}
The proof can be found in the Appendix \ref{sec: appendix}.

Theorem \ref{the: RDQRCP} provides the existence of a randomized QR decomposition with column pivoting in the standard part of dual matrices. Algorithm \ref{alg: RDQRCP} presents its pseudocode.
\begin{algorithm}[htb]
\DontPrintSemicolon
    \KwInput{$A = A_{s}+A_{i}\epsilon \in \mathbb{DR}^{m \times n}$}

     Compute the RQRCP of $A_{s}P_{s} = Q_{s}R_{s}$ in Algorithm \ref{alg: Single-Sample Randomized QRCP} $(Q_{s}\in \mathbb{R}^{m \times k},R_{s}\in \mathbb{R}^{k \times n}, P_{s} \in \mathbb{R}^{n \times n})$.

    \While{$(I_{m}-Q_{s}Q_{s}^{\top})A_{i}P_{s}(I_{n}-R_{s}^{\dagger}R_{s}) = O_{m\times n}$}{

        Set $P = zeros(k,k)$ and $B = Q_{s}^{\top}A_{i}P_{s}$.

        Set $i = 1$.

        $P(2:k,1) = B(2:k,1)/R_{s}(1,1)$.

        \While{$i < k-1 $}{

            $i = i+1$.

            $P(i+1:k,i) = (B(i+1:k,i)-\sum\limits_{t=1}^{i-1}R_{s}(t,i)P(i+1:k,t)/R_{s}(i,i)$.

            $P(i,i+1:k) = -P(i+1:k,i)^{\top}$.

        }

        Set $Q_{i} = (I_{m}-Q_{s}Q_{s}^{\top})A_{i}P_{s}R_{s}^{\dagger} + Q_{s}P$.

        Set $R_{i}  = B -PR_{s}$.

    }

    \KwOutput{$Q = Q_{s}+Q_{i}\epsilon \in \mathbb{DR}^{m \times k}$ is a dual matrix with orthogonal columns, $R = R_{s}+R_{i}\epsilon \in \mathbb{DR}^{k \times n}$ is an upper trapezoidal  dual matrix, and $P_{s} \in \mathbb{R}^{n \times n}$ is a permutation matrix.}
    \caption{Randomized QR Decomposition with Column Pivoting in the standard part of Dual Matrices (RDQRCP)}
    \label{alg: RDQRCP}
\end{algorithm}

\section{Perturbation Analysis of the Q-factor}\label{sec: Perturbation Analysi}
There has been extensive research \cite{stewart1993perturbation, sun1995perturbation} on perturbation analysis for the thin QR decomposition of a matrix $A\in\mathbb{R}^{m \times n}$ with full column rank, where $A = QR, Q\in \mathbb{R}^{m \times n}$ and $R \in \mathbb{R}^{n \times n}$.
Suppose that $\Delta A \in \mathbb{R}^{m \times n}$ is a small perturbation matrix such that $A + \Delta A$ has full column rank, then $A + \Delta A$ has a unique QR factorization
\begin{equation}\label{equ: A+dA = Q+dQ R+dR}
    A + \Delta A = (Q + \Delta Q)(R + \Delta R),
\end{equation}
where $Q + \Delta Q$ has orthonormal columns and $R + \Delta R$ is upper triangular with positive diagonal entries.
Perturbation analysis for the Q-factor in the QR factorization has received extensive attention in the literature \cite{stewart1993perturbation, sun1995perturbation}.
The infinitesimal part within dual numbers possesses highly distinctive properties
\begin{equation}
    \epsilon^{2} = 0, \qquad   (\epsilon \neq 0).
\end{equation}
Hence, if we view the decomposition of dual number matrices as the decomposition of matrices with perturbed terms, then what we obtain is an accurate first-order approximation.
We substitute $A = A_{s}+\epsilon A_{i}\in \mathbb{DR}^{m \times n}$ for $A+\Delta A$ in (\ref{equ: A+dA = Q+dQ R+dR}), $Q = Q_{s}+\epsilon Q_{i}\in \mathbb{DR}^{m \times n}$ for $Q+\Delta Q$, and $R = R_{s}+\epsilon R_{i}\in \mathbb{DR}^{n \times n}$ for $R+\Delta R$.
Hence, using dual numbers, we can obtain the explicitly accurate first-order perturbation \cite{Greenbaum2020first} of the QR factorization,
\begin{equation}\label{equ: As+Ai = (Qs+Qi)(Rs+Ri)}
    A_{s}+\epsilon A_{i} = (Q_{s}+\epsilon Q_{i})(R_{s}+\epsilon R_{i}).
\end{equation}
This allows us to view the QR decomposition of dual matrices from the perspective of
perturbation analysis.
For the three types of perturbation analyses previously mentioned and given conditions, it is evident that we can satisfy all of them.
Stewart \cite{stewart1993perturbation} initially provided perturbation bounds for the Q-factor in QR decomposition under componentwise computation,
\begin{equation}\label{equ: QiF <= 1+sqrt2AsAi}
    \|Q_{i}\|_{F} \lesssim (1+\sqrt{2})\|A_{s}^{\dagger}\|_{2}\|A_{i}\|_{F}.
\end{equation}
Subsequently, Sun \cite{sun1995perturbation} improved rigorous perturbation bounds for the Q- factor $Q_{i}$ when $A_{i}$ is given,
\begin{equation}\label{equ: QiF <= sqrt2AsAi}
    \|Q_{i}\|_{F} \lesssim \sqrt{2}\|A_{s}^{\dagger}\|_{2}\|A_{i}\|_{F}.
\end{equation}
Now, we utilize the examples  in \cite{sun1995perturbation} to  test whether our tDQR algorithm satisfies the perturbation upper bounds as indicated in (\ref{equ: QiF <= sqrt2AsAi}). Let
    \begin{equation}\label{equ: example_sun}
        A_{s} = \begin{bmatrix}
            1&-2&1&2&3\\
            0&2&4&1&-5\\
            0&0&3&-1&2\\
            0&0&0&4&1\\
            0&0&0&0&5\\
            0&0&0&0&0\\
            0&0&0&0&0\\
            0&0&0&0&0
        \end{bmatrix}, \quad  A_{i} = \tau\begin{bmatrix}
            0.2&-0.5&0.3&0.1&0.4\\
            -0.1&0.4&0.1&-0.3&0.2\\
            0.5&0.7&-0.2&0.1&0.6\\
            0.3&-0.6&0.1&-0.1&0.2\\
            0.2&0.1&0.7&0.3&-0.4\\
            0.4&0.8&-0.2&0.1&0.3\\
            0.6&-0.1&-0.5&0.1&-0.2\\
            0.1&-0.3&0.2&0.6&0.7
        \end{bmatrix},
    \end{equation}
where $\tau > 0 $.
Table \ref{tab: Perturbation Errors in tDQR Decomposition and tQR Decomposition } illustrates the perturbation errors obtained through the MATLAB QR algorithm and our tDQR decomposition in Algorithm \ref{alg: thin QR Decomposition of Dual Matrix} for different small values of $\tau$.
Clearly, through numerical experiments, we can observe that the perturbation errors obtained using dual matrices are smaller and adhere to the perturbation bounds given by (\ref{equ: QiF <= sqrt2AsAi}).
This further supports the notion that dual matrix QR decomposition is an explicitly accurate first-order approximation with improved orthogonality.
\begin{table}[htb]
\caption{Perturbation Errors in tDQR Decomposition and QR Decomposition}
\label{tab: Perturbation Errors in tDQR Decomposition and tQR Decomposition }
\centering
\begin{tabular}{|c|c|c|c|c|}
\hline
$\tau$                                         & 1e-1       & 1e-2       & 1e-5           & 1e-8           \\ \hline
$\|A_{i}\|_{F}$                                & 0.24310492 & 0.02431049 & 2.43104916e-05 & 2.43104915e-08 \\ \hline
$\|\Delta Q\|_{F}$                             & 0.33558049 & 0.03162668 & 3.13819444e-05 & 3.13816983e-08 \\ \hline
$\|Q_{i}\|_{F}$                                & 0.31381698 & 0.03138170 & 3.13816980e-05 & 3.13816980e-08 \\ \hline
$\sqrt{2}\|A_{s}^{\dagger}\|_{2}\|A_{i}\|_{F}$ & 1.04034046 & 0.10403405 & 1.04034046e-04 & 1.04034045e-07 \\ \hline
\end{tabular}
\end{table}

\section{Numerical Experiments}\label{sec: Numerical Exiperment}
In this section, we compare the various proposed QR algorithms for the dual matrix and assess their respective suitability for different scenarios.
Next, we utilize the QR decomposition of dual matrices to construct and prove the existence of dual Moore-Penrose generalized inverse (DMPGI).
Numerical examples demonstrate that our approach achieves higher precision compared to previous methods.
Subsequently, we verify the effectiveness of QR algorithms for the dual matrix in traveling wave detection for composite waves.
Finally, we apply the proposed QR algorithm of the dual matrix to traveling wave detection in large-scale brain functional magnetic resonance imaging (fMRI) data in the field of neurosciences, significantly enhancing the efficiency relative to previous algorithms.
The code was implemented using Matlab R2023b and all experiments were conducted on a personal computer equipped with an M2 pro Apple processor and 16GB of RAM, running macOS Ventura 13.3.

\subsection{Numerical Comparison}\label{sec: Numerical Comparsion}
In this section, we compare the runtime of the DQR, the DQRCP, the tDQ, and the tDQRCP decompositions of a dual matrix $A\in \mathbb{DR}^{m \times n}$ under various scenarios where $m > n, m = n$, and $m < n$, to assess their performance at different scales.
We construct a GIID dual matrix $A$ as the synthetic data for evaluating and comparing.
The assessment criteria comprise runtime measurements due to the absence of uniform norms, rendering computational error comparisons challenging. To ensure the reliability of our findings, each algorithm is executed 10 times, resulting in robust results.

\begin{table}[htb]
\caption{Comparison between DQR decomposition and DQRCP decomposition}
\label{tab: Comparison between DQR DQRCP }
\centering
\begin{tabular}{|c|c|c|}
\hline
\diagbox{Size}{Time(s)}{Algorithm}                      & DQR               & DQRCP             \\ \hline
$500 \times 500$    & 0.1267(1.1504e-4) & 0.1344(3.5918e-5) \\ \hline
$500 \times 1000$   & 0.1456(2.3226e-4) & 0.1748(1.9633e-4) \\ \hline
$1000 \times 500$   & 0.2727(0.0011)    & 0.2815(6.240e-4)  \\ \hline
$2500 \times 2500$  & 5.7008(0.1042)    & 5.7608(0.0054)    \\ \hline
$2500 \times 5000$  & 6.9206(0.0210)    & 8.0753(0.0622)    \\ \hline
$5000 \times 2500$  & 16.6889(0.4629)   & 17.4049(0.1718)   \\ \hline
$4000 \times 4000$  & 19.1934(1.4439)   & 22.3248(2.2537)   \\ \hline
$4000 \times 8000$  & 21.4972(0.6180)   & 28.0318(0.6010)   \\ \hline
$8000 \times 4000$  & 70.5938(0.9510)   & 73.6737(0.9375)   \\ \hline
$5000 \times 5000$  & 33.3352(1.4165)   & 38.3178(1.2820)   \\ \hline
$5000 \times 10000$ & 38.3174(2.1581)   & 51.5523(1.2879)   \\ \hline
$10000 \times 5000$ & 123.6001(18.5784) & 133.7725(15.5367) \\ \hline
\end{tabular}
\end{table}

\begin{table}[htb]
\caption{Comparison between tDQR decomposition and tDQRCP decomposition}
\label{tab: Comparison between thin DQR DQRCP }
\centering
\begin{tabular}{|c|c|c|}
\hline
\diagbox{Size}{Time(s)}{Algorithm}                     & tDQR              & tDQRCP            \\ \hline
$500 \times 500$    & 0.1319(2.9287e-4) & 0.1422(6.0144e-5) \\ \hline
$500 \times 1000$   & 1.9193(0.0134)    & 2.9907(0.0430)    \\ \hline
$1000 \times 500$   & 0.1849(7.980e-4)  & 0.1808(3.636e-4)  \\ \hline
$2500 \times 2500$  & 5.7371(0.0094)    & 6.0572(0.0162)    \\ \hline
$2500 \times 5000$  & 18.1065(4.5056)   & 20.0397(1.7046)   \\ \hline
$5000 \times 2500$  & 8.3085(0.1454)    & 8.6551(0.1427)    \\ \hline
$4000 \times 4000$  & 20.4515(1.1058)   & 22.6906(1.5815)   \\ \hline
$4000 \times 8000$  & 62.7255(38.3143)  & 67.9553(20.5681)  \\ \hline
$8000 \times 4000$  & 28.2839(1.1321)   & 31.9453(0.9432)   \\ \hline
$5000 \times 5000$  & 34.1231(3.4861)   & 39.2982(1.8496)   \\ \hline
$5000 \times 10000$ & 109.0411(22.0643) & 120.6198(20.0132) \\ \hline
$10000 \times 5000$ & 48.1549(7.5332)   & 55.9513(6.7111)   \\ \hline
\end{tabular}
\end{table}

Tables \ref{tab: Comparison between DQR DQRCP }  and \ref{tab: Comparison between thin DQR DQRCP } present a time comparison for the DQR, the DQRCP, the tDQR, and the tDQRCP decompositions for dual matrices of varying dimensions, with variance values provided in parentheses.
It is evident that QR decompositions with column pivoting consistently exhibit slower runtime than those without column pivoting.
However, they offer enhanced numerical stability.
When considering dual matrices of size $m \times  n$, we observe that for $m > n$, both DQR decomposition and DQRCP decomposition outperform the tDQR decomposition and the tDQRCP decomposition in terms of speed.
Conversely, when $m < n$, the situation is reversed, thus confirming our hypothesis.
When $m = n$, it is noticeable that (\ref{equ: Qi=QsP}) lacks the presence of the term $(I_{m}-Q_{s}Q_{s}^{\top})A_{i}R_{s}^{-1}$ compared to (\ref{equ: thin Qi Ri expression}).
However, in cases where $m$ and $n$ are equal, and the dual matrix has full column rank, then $(I_{m}-Q_{s}Q_{s}^{\top})A_{i}R_{s}^{-1} = O_{m \times n}$.
Nonetheless, thin QR decomposition still computes this term, leading to increased runtime.

Next, we aim to construct a low-rank dual matrix $A \in \mathbb{DR}^{m \times n}$,
where $m > n$. Let $L$ be a GIID $m \times r$ dual matrix, and $R$ be a GIID $r\times n$ dual matrix, where $r = round(n/10)$.
For the synthetic data $A = LR$ that we have constructed, we compare the efficiency of the low-rank decomposition using the  tDQRCP decomposition and the RDQRCP decomposition, with the target rank $k$ set to 10.
To ensure the reliability of our findings, each algorithm is executed 50 times, resulting in robust results.

\begin{table}[htb]
\caption{Comparison between tDQRCP decomposition and RDQRCP decomposition}
\label{tab: Comparison between tDQRCR RDQRCP }
\centering
\begin{tabular}{|c|c|c|}
\hline
\diagbox{Size}{Time(s)}{Algorithm}                    & tDQRCP            & RDQRCP            \\ \hline
$1000 \times 200$  & 0.0546(2.1961e-4) & 0.0100(3.2969e-5) \\ \hline
$2000 \times 400$  & 0.2088(0.0017)    & 0.0386(3.0649e-4) \\ \hline
$4000 \times 1000$ & 1.4337(0.0114)    & 0.1557(2.1831e-4) \\ \hline
$8000 \times 2000$ & 7.3579(0.0973)    & 0.6936(0.0058)    \\ \hline
\end{tabular}
\end{table}

Table \ref{tab: Comparison between tDQRCR RDQRCP } presents a time comparison for the tDQRCP decomposition and the RDQRCP decomposition for dual matrices of varying dimensions, with variance values provided in parentheses.
Clearly, in the low-rank problem, the RDQRCP decomposition exhibits faster speed and superior stability, with the potential to handle larger-scale data within memory constraints.

\subsection{Dual Moore-Penrose Generalized Inverse}\label{sec: dual Moore-Penrose generalized inverse}

In previous research on dual matrices, many scholars have investigated the theory of dual Moore-Penrose generalized inverse (DMPGI) \cite{angeles2012dual,udwadia2021does,udwadia2020all,wang2021characterizations}.
We utilize the QR decomposition of dual matrices to compute DMPGI as an example.
Similar to how Definition \ref{def: Moore Penrose Inverse} defines the Moore-Penrose pseudoinverse, we can define the DMPGI accordingly.
\begin{definition}\label{def: DMPGI}
    Given a dual matrix $A\in \mathbb{DR}^{m \times n}$, its dual Moore-Penrose generalized inverse (DMPGI), denoted as $A^{\dagger}$ which is a dual matrix that satisfies the following four properties,
    \begin{equation}
        AA^{\dagger}A =A,\quad A^{\dagger}AA^{\dagger} = A^{\dagger}, \quad (AA^{\dagger})^{\top} = AA^{\dagger}, \quad (A^{\dagger}A)^{\top} = A^{\dagger}A.
    \end{equation}
\end{definition}
In the following theorem, we establish the relationship between the existence of DMPGI and QR decomposition of the dual matrix.
\begin{theorem}\label{the: DMPGI}
    Given a dual matrix $A \in \mathbb{DR}^{m \times n}$ where $A_{s}\in \mathbb{R}^{m \times n}$ is full column rank. Then the DMPGI $A^{\dagger}$ of $A$  exists, and
    \begin{equation}
        A^{\dagger}  = R_{s}^{-1}Q_{s}^{\top} + (R_{s}^{-1}Q_{i}^{\top}-R_{s}^{-1}R_{i}R_{s}^{-1}Q_{s}^{\top})\epsilon
    \end{equation}
    where $A = (Q_{s}+Q_{i}\epsilon)(R_{s}+R_{i}\epsilon)$ is the thin QR decomposition of $A$.
\end{theorem}
\begin{proof}
    According to Theorem \ref{the: thin QR decomposition fo dual matrices}, since $A_{s}$ is full column rank, the thin QR decomposition of $A = QR = (Q_{s}+Q_{i}\epsilon)(R_{s}+R_{i}\epsilon) = Q_{s}R_{s} + (Q_{s}R_{i}+Q_{i}R_{s})\epsilon$ exists.
    Then the DMPGI of $A$ exists, taking the form $A^{\dagger} = R^{\dagger}Q^{\top}$, satisfying Definition \ref{def: DMPGI}.
    Through rigorous computation, the matrix $R^{\dagger} = R_{s}^{-1}-R_{s}^{-1}R_{i}R_{s}^{-1} \epsilon$ can be derived, from which we can obtain
    \begin{equation}
        \begin{aligned}
            A^{\dagger} & = R^{\dagger}Q^{\top} \\
            & = (R_{s}^{-1}-R_{s}^{-1}R_{i}R_{s}^{-1} \epsilon)(Q_{s}^{\top}+Q_{i}^{\top} \epsilon) \\
            & = R_{s}^{-1}Q_{s}^{\top} + (R_{s}^{-1}Q_{i}^{\top}-R_{s}^{-1}R_{i}R_{s}^{-1}Q_{s}^{\top})\epsilon.
        \end{aligned}
    \end{equation}
\end{proof}
Based on the results provided by Theorem \ref{the: DMPGI}, we can compute the DMPGI using the QR decomposition of the dual matrix.
Utilizing the examples in reference \cite{pennestri2007linear}, let's consider the dual matrix
\begin{equation}
    A_{1} = \begin{bmatrix}
        1&3\\
        9&22\\
        4&4
    \end{bmatrix}+\epsilon\begin{bmatrix}
        4&0\\
        2&4\\
        4&1
    \end{bmatrix},\quad A_{2} = \begin{bmatrix}
        1&3&4\\
        9&22&4
    \end{bmatrix}+\epsilon\begin{bmatrix}
        4&0&1\\
        2&4&4
    \end{bmatrix}.
\end{equation}
By the computation based on Theorem \ref{the: DMPGI}, we determine that the DMPGIs of $A_{1}$ and $A_2$ are respectively
\begin{equation}
    A_{1}^{\dagger} = \begin{bmatrix}
         -0.0508 &  -0.0691  &  0.4182\\
    0.0276  &  0.0727  & -0.1704
    \end{bmatrix}+\epsilon \begin{bmatrix}
        0.8221  & -0.0350 &  -0.4596 \\
   -0.3493  &  0.0117  &  0.1675
    \end{bmatrix}
\end{equation}
and
\begin{equation}
    A_{2}^{\dagger} = \begin{bmatrix}
        -0.0349  &  0.0210\\
   -0.0379  &  0.0438 \\
    0.2872  & -0.0381 \\
    \end{bmatrix}+\epsilon \begin{bmatrix}
        -0.0143   & 0.0002\\
   -0.0350  & -0.0011\\
   -0.0071  & -0.0107\\
    \end{bmatrix}.
\end{equation}

\begin{remark}\label{rem: DMPGI}
    Note that the DMPGI of $A_{1}$ calculated here differs from the result in reference \cite{pennestri2007linear}.
    Through simple numerical verification, we ascertain that our method provides higher accuracy.
    For the DMPGI of $A_{2}$, our computation matches the result in reference \cite{pennestri2007linear}, indicating that with the assistance of QR decomposition, we achieve higher precision in DMPGI.
    Our approach offers an effective, concise, and precise method for computing DMPGI.
\end{remark}

\subsection{Simulations of Standing and Traveling Waves}\label{sec: Simulations of Standing and Traveling Waves.}
In the past, numerous theoretical explorations have paved the way for research into the application aspects of dual numbers.
Eduard Study \cite{study1903geometrie} introduced the concept of dual angles when determining the relative positions of two skew lines in three-dimensional space, where the angle is defined as the standard part and the distance as the infinitesimal part.
Similarly, when dealing with raw time series data, we regard the standard part as the observed data and the infinitesimal part as its derivative or first-order difference, which naturally identifies standing and traveling waves in the simulations.
Consider a spatiotemporal propagating pattern \cite{feeny2008complex}
 \begin{equation}\label{equ: x = 2e cos c -sin d}
     \textbf{x}(t) = 2 e^{\gamma t} [\cos{(\omega t)}\textbf{c} - \sin{ (\omega t)} \textbf{d}],
 \end{equation}
where $\textbf{x}(t)$ represents a real vector of particle positions at the time $t$,  $\gamma$, $\omega$ are real scalars representing the exponential decay rate and angular frequency, and $\textbf{c}$, $\textbf{d}$ are real vectors representing two modes of oscillation.
Overall, (\ref{equ: x = 2e cos c -sin d}) demonstrates the circularity and continuity of propagation between $\textbf{c}$ and $\textbf{d}$.
When $\textbf{c}=\textbf{d}$, (\ref{equ: x = 2e cos c -sin d}) represents a `standing wave' mode oscillation of rank-1.
When $\textbf{c}\neq \textbf{d}$, (\ref{equ: x = 2e cos c -sin d}) represents a `traveling wave' mode oscillation of rank-2 which behaves as a wavy motion.
In previous research \cite{feeny2008complex}, wave propagation percentages of a wave were measured using an index, without distinguishing the exact traveling wave from the original wave.
Wei, Ding, and Wei \cite{wei2024singular} introduced an effective method for traveling wave identification and extraction, including the determination of its propagation path.
Now, we build upon their ideas and explore whether the QR decompositions of dual matrices proposed in this paper can be effectively utilized for traveling wave identification and whether they offer computational advantages.
Here, we employ proper orthogonal decomposition to investigate wave behavior.
Therefore, our focus lies in the $Q$ component of the QR decomposition for the dual matrix.

\begin{figure}[htp]
    \centering
    \includegraphics[width = 0.8\textwidth,height = 1.1\textwidth]{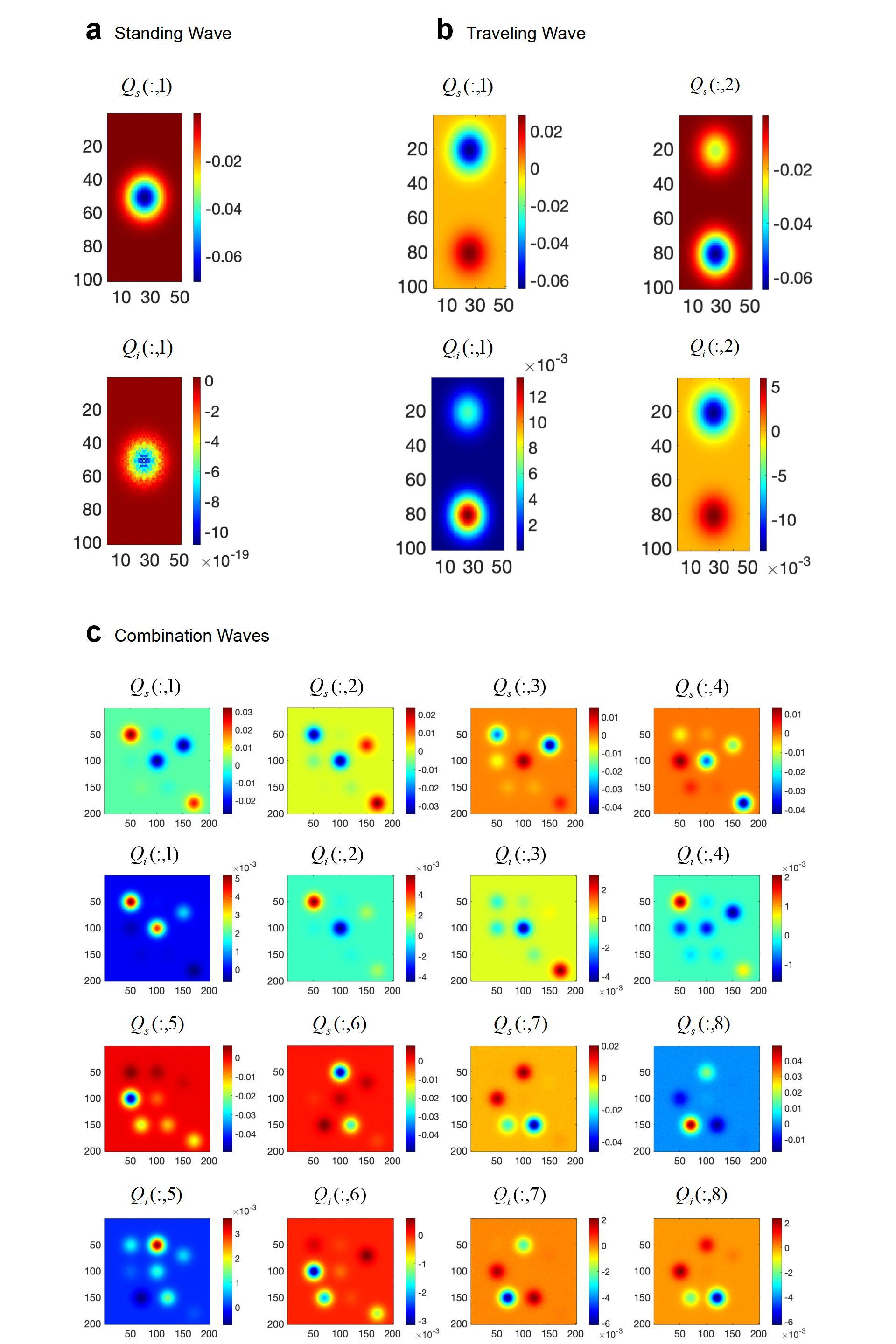}
    \caption{The simulations of pure standing waves exhibit the center positions in the standard part of $Q$ and small values in the infinitesimal part. Simulations of pure traveling waves reveal paired similarities between the standard and infinitesimal parts of $Q$. Leveraging the QR decomposition of the dual matrix and the specific properties of traveling waves, we accurately identify traveling waves within a combination wave composed of four standing waves and two traveling waves.}
    \label{fig:  wave}
\end{figure}

Now, we consider the spatiotemporal propagation patterns in (\ref{equ: x = 2e cos c -sin d}) and construct a two-dimensional grid constructed as a $50 \times 100$ array of pixels by Gaussian waves.
We arrange $m$ particles in rows and $n$ time points in columns, allowing us to construct an $m\times n$ ensemble matrix $X_{s}$.
Here, we set $\textbf{c}$ and $\textbf{d}$ as two-dimensional Gaussian curves.
By taking the derivative of $X_{s}$ with respect to time, we obtain $X_{i}$.
By treating $X_{s}$ as the standard part and $X_{i}$ as the infinitesimal part, we can construct the dual matrix $X$.
When $c=d$, Figure \ref{fig:  wave}(a) illustrates the standard part $Q_{s}$ and infinitesimal part $Q_{i}$ of the dual matrix $X$ after the QR decomposition.
This is clearly a standing wave, with the center point at $(25, 50)$.
When $c\neq d$, Figure \ref{fig:  wave}(b) displays the standard part $Q_{s}$ and infinitesimal part $Q_{i}$ of the dual matrix $X$ after QR decomposition, with center points at $(25, 20)$ and $(25, 80)$.
It's noteworthy that signals $Q_{s}(:,1), Q_{i}(:,2), Q_{s}(:,2)$, and $Q_{i}(:,1)$  exhibit paired similarities, where one is homologous and the other is heterologous.
This property can be succinctly explained under the rank-2 truncated QR decomposition.
First, due to the construction of $X$, we have the range spaces $\textbf{Ran}(X_{i}) \subseteq \textbf{Ran}(X_{s}) $, implying the existence of a matrix $M\in \mathbb{R}^{n \times n}$ such that $X_{s} = X_{i}M$.
Consequently, we have
\begin{equation}\label{equ: I-QQAi = o}
    (I_{m}-Q_{s}Q_{s}^{\top})A_{i}=(I_{m}-Q_{s}Q_{s}^{\top})A_{s}M = (I_{m}-Q_{s}Q_{s}^{\top})Q_{s}R_{s}M = O_{m\times n}.
\end{equation}
For a $2\times 2$ skew-symmetric matrix $P = \begin{bmatrix}
    0 & p_{12}\\
    -p_{12} & 0
\end{bmatrix}$, it is evident that
\begin{equation}\label{equ: Qi = QsP = Qs0p12-p120}
    Q_{i}= Q_{s}P = \begin{bmatrix}
        Q_{s}(:,1)& Q_{s}(:,2)
    \end{bmatrix}\begin{bmatrix}
    0 & p_{12}\\
    -p_{12} & 0
\end{bmatrix} = \begin{bmatrix}
    -p_{12}Q_{s}(:,2) & p_{12}Q_{s}(:,1)
\end{bmatrix}.
\end{equation}
By utilizing (\ref{equ: I-QQAi = o}) and (\ref{equ: Qi = QsP = Qs0p12-p120}) to analyze the $Q$ components of the QR decompositions proposed in this paper, it becomes apparent that paired similarities are an inevitable occurrence.

To validate the effectiveness of traveling wave identification, we construct a combination wave and attempt to separate different signals.
We combine four standing waves and two traveling waves with different frequencies and weights, each with the same standard deviation $\sigma = 1$.
Additionally, we introduce Gaussian noise with a peak value of $10^{-3}$.
Using the same methodology as before, we treat the data matrix of the combination wave as the standard part and its first-order derivative as the infinitesimal part to generate the dual matrix $Y$.
We perform a QR decomposition on $Y$ and utilize $Q$ to separate the six different waves.
Figure \ref{fig:  wave}(c) illustrates the standard and infinitesimal parts of the first four columns of $Q$, corresponding to four weighted standing waves.
The centers of these four standing waves (50, 50), (100, 100), (150, 70), and (170, 180) represent the wave peaks.
Similarly, in the standard and infinitesimal parts of columns 5 to 8 of $Q$, we can observe two pairs of traveling waves.
In one pair, (50, 100) and (100, 50) represent the wave peaks, while in the other pair, (70, 150) and (120, 150) serve as the wave peaks.

In summary, our QR decomposition of the dual matrix provides an effective approach for traveling wave identification through proper orthogonal decomposition.
Moreover, it is computationally straightforward.

\subsection{Traveling Wave Identification in the Brain}\label{sec: Traveling Wave Identification In the Brain}

Building upon prior research regarding the application of dual matrices in traveling wave identification for time series data in the brain, we hypothesize that performing a proper orthogonal decomposition on them yields superior results.
Subsequently, we employ the previously proposed QR decomposition algorithm for dual matrices to test our hypothesis.

In our subsequent experiments, we made use of the Human Connectome Project (HCP) database \cite{van2013wu}, which comprises high-resolution 3D magnetic resonance scans of healthy adults aged between 22 to 35 years.
This database includes task-state functional magnetic resonance imaging (tfMRI) images, a pivotal imaging modality in medical analysis.
These tfMRI images were acquired while the subjects engaged in various tasks designed to stimulate cortical and subcortical networks.
Each tfMRI scan encompasses two phases: one with right-to-left phase encoding and the other with left-to-right phase encoding.
Achieving an in-plane field of view (FOV) rotation was accomplished by inverting both the readout (RO) and phase encoding (PE) gradient polarity \footnote{\url{https://www.humanconnectome.org/storage/app/media/documentation/s1200/HCP_S1200_Release_Reference_Manual.pdf}}.

Building upon the theoretical groundwork laid out in Section \ref{sec: Simulations of Standing and Traveling Waves.} regarding traveling wave identification by proper orthogonal decomposition, our present objective is to delve into the associations between this methodology and empirical brain responses within functional brain regions.
Specifically, we concentrate on assessing the applicability of traveling wave identification within brain regions associated with typical language processing tasks encompassing all seven types, including semantic and phonological processing tasks, as initially developed by Binder {\it et al.} \cite{binder2011mapping}.
In this experimental setup, Binder {\it et al.} selected four blocks of a story comprehension task and four blocks of an arithmetic task.
The story comprehension task involved the presentation of short auditory narratives consisting of 5-9 sentences, followed by binary-choice questions related to the central theme of the story.
Conversely, the arithmetic task required participants to engage in addition and subtraction calculations during auditory experiments.

For our investigation, we considered a randomly chosen participant from the aforementioned experiment.
The corresponding task-state functional magnetic resonance imaging (tfMRI) sequence data associated with the language processing task is stored as a 4-dimensional voxel-based image, encompassing $91 \times 109\times 91$ three-dimensional spatial locations and 316 frames.
The voxel size is set at 2mm, and the repetition time (TR) is established at 0.72 seconds.
Data preprocessing assumes a pivotal role in maintaining substantial information while effectively compressing the image matrix.
We initiated temporal filtering of the tfMRI signals, employing a Butterworth band-pass zero-phase filter in the frequency range of 0.01-0.1 Hz.
This step aligns the signals with the conventional low-frequency range.
Subsequently, we applied spatial smoothing, employing a 3-dimensional Gaussian low-pass filter with a half-maximum width of 5mm, involving convolution and exclusion of NANs.
To ensure uniformity in voxel size across different frames, we implemented a Z-score transformation on the BOLD time series of all voxels. This transformation ensures a mean of zero and unit variance, maintaining consistency in the data.

 \begin{figure}[htb]
   \centering
    \begin{minipage}[c]{0.48\linewidth}
    \centering
    \includegraphics[width=1\linewidth]{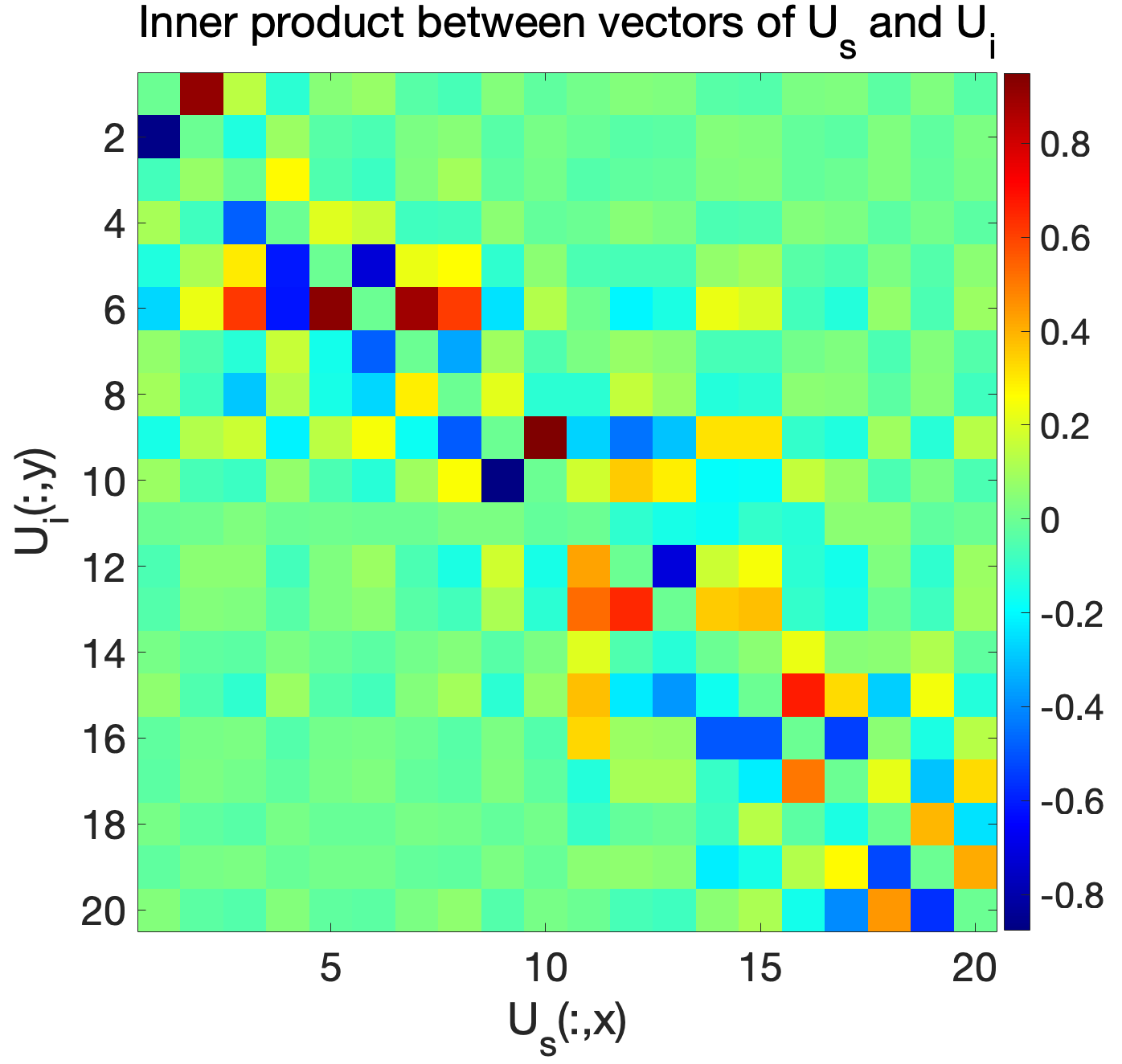}
    \end{minipage}
    \begin{minipage}[c]{0.48\linewidth}
    \centering
    \includegraphics[width=1\linewidth]{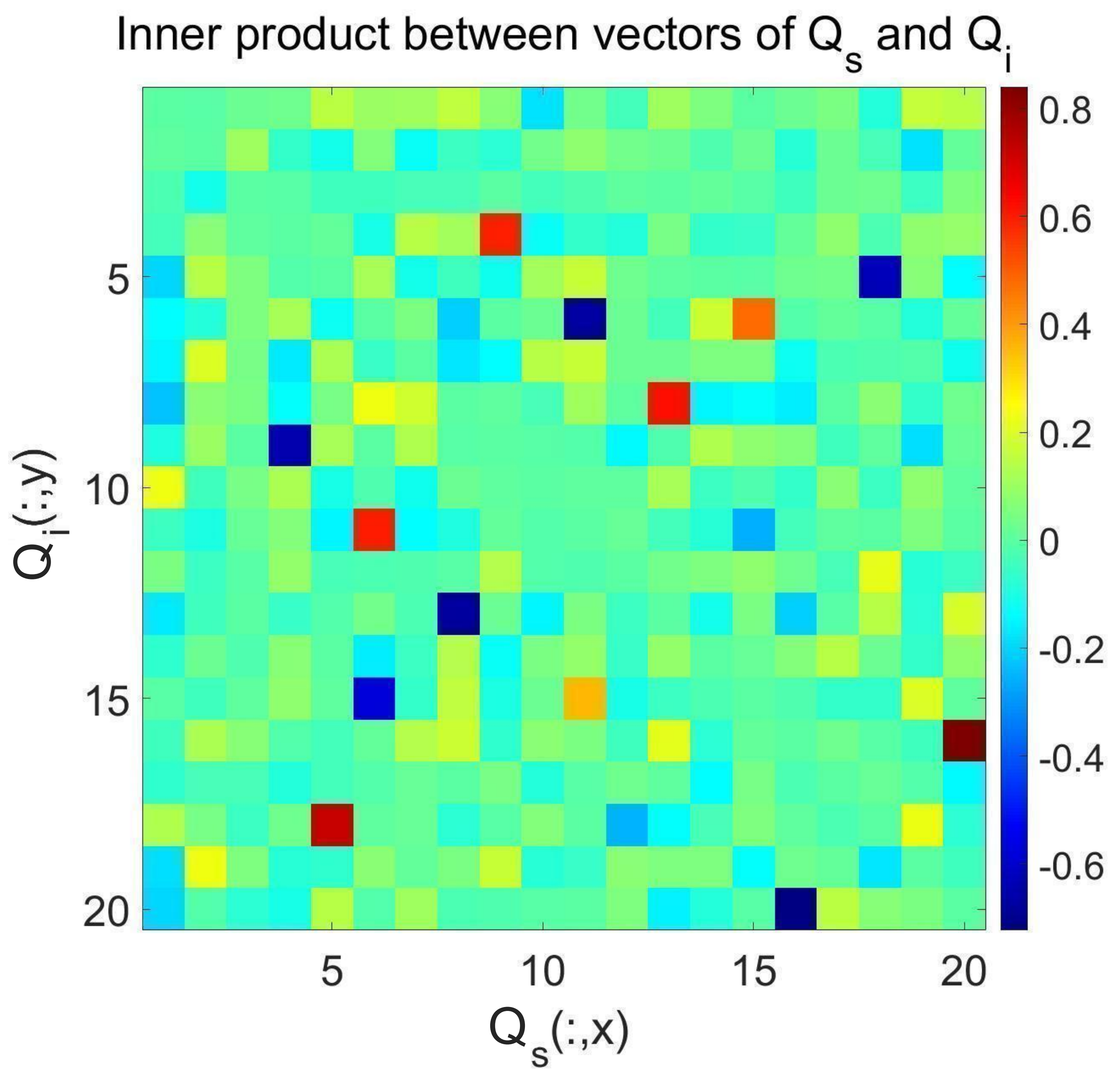}
    \end{minipage}
   \caption{Inner product between vectors of the standard part and infinitesimal part in the CDSVD and the RDQRCP decomposition}
   \label{fig: Inner}
 \end{figure}

Following this, our primary aim is to derive a dual matrix from the original brain sequence data matrix and affirm the coherence between areas displaying significant extracted traveling waves and the respective functional cortical activities in the brain.
To achieve this, the initial three spatial dimensions are transformed into column vectors for each frame, and the fourth temporal dimension is sequentially appended.
This operation culminates in the creation of a real matrix.
This real matrix constitutes what we refer to as the standard component of the dual matrix $D$.
Similarly, the infinitesimal component is acquired by performing first-order temporal differentiation.

By utilizing the previously proposed RDQRCP in Algorithm \ref{alg: RDQRCP} of the dual matrix $D$, we easily obtain the desired proper orthogonal decomposition. Furthermore, compared to other researchers' compact dual singular value decomposition (CDSVD) \cite{wei2024singular}, our orthogonality is more pronounced.
While CDSVD requires 14.661 seconds, we only need 5.660  seconds, providing a significant advantage in terms of processing time.

In Figure \ref{fig: Inner}, the utilization of CDSVD \cite{wei2024singular} and the RDQRCP is presented in proper orthogonal decomposition, showcasing the results of inner products between the standard part and infinitesimal part.
More specifically, within the QR decomposition for dual matrices, the inner product computed between $Q_{s}(:, 16)$ and $Q_{i}(:, 20)$ yields a substantial value of 0.840.
Similarly, for $Q_{i}(:, 20)$ and $Q_{s}(:, 16)$, the inner product registers -0.723.
These calculations unveil the presence of a conspicuous traveling wave, characterized by the dominant principal components residing in the standard part.
Subsequent to this discovery, the rank-2 matrix, originating from the second and third principal components, is meticulously rescaled to its original dimensions.
This rescaled matrix is then subjected to spatial projection onto the brain's anatomical
structure using the `BrainNet' toolkit within the MATLAB environment \cite{xia2013brainnet}.
Through a sequential concatenation of these brain images, a dynamic video representation is created.
This video provides a visually striking means of observing the temporal dynamics of brain activity across distinct regions, as depicted through the dynamic variations in color representation.

\begin{figure}[htp]
    \centering
    \includegraphics[width = 0.98\textwidth]{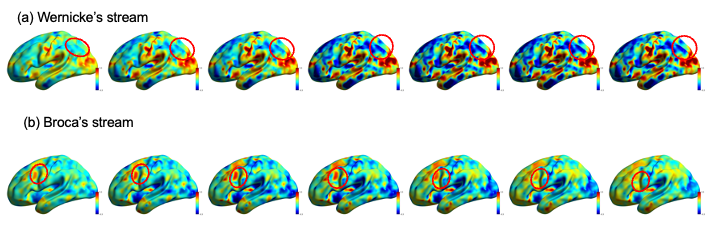}
    \caption{Paths of traveling waves including Wernicke’s stream and Broca’s stream.}
    \label{fig: area}
\end{figure}
In Figure \ref{fig: area} (a), we can discern a substantial signal propagation.
This module represents a language task in progress, where the wave originates from Brodmann area 39 (angular gyrus) and terminates at Brodmann area 40 (supramarginal gyrus) \cite{sakurai2017brodmann}.
Both of these regions are intricately linked with Wernicke's area, a pivotal node in language processing, particularly in the comprehension of speech.
This specific cerebral region plays a critical role in the interpretation and generation of language, especially in speech comprehension.
It enables us to adeptly grasp spoken language from others and articulate our own thoughts and emotions.
Additionally, the prominence of this wave in the left hemisphere, as opposed to the right hemisphere, underscores the left hemisphere's primary role in language processing.

Simultaneously, in Figure \ref{fig: area} (b), we also discern the propagation of the wave from Brodmann area 44 (pars opercularis) to Brodmann area 45 (Pars triangularis), both of which are constituents of Broca's area \cite{friederici2017language}.
This region plays a pivotal role in language processing, encompassing tasks related to the formulation and organization of language expression, as well as the processing of grammatical structures and syntax.
Lesions or impairments in this area can lead to Broca's aphasia, characterized by challenges in sentence production and reduced linguistic fluency, while language comprehension abilities remain largely intact.
Additionally, it is implicated in language memory functions, particularly short-term memory processes.

To summarize, our findings have delineated two distinct wave patterns that correspond to well-established brain functional circuits, thus corroborating prior research in the field of cognitive neuroscience.
In essence, this serves as a robust validation mechanism that reinforces the established conclusions within this domain.

\small{
\appendix

\label{appendix}
\section{Proof of Theorem 3.2 and Theorem 3.3}\label{sec: appendix}

The proof of Theorem \ref{the: thin QR decomposition fo dual matrices}. \\
\begin{proof}
	Suppose the thin QR decomposition of $A$ exists, then we have that
	\begin{equation}\label{equ: thin As=QsRs Ai = QsRi + QiRs}
		\begin{cases}
			A_{s} = Q_{s}R_{s}\\
			A_{i} = Q_{s}R_{i}+Q_{i}R_{s}
		\end{cases}
	\end{equation}
	through simple computation. By treating $R_{i}$ as $X$ and $Q_{i}$ as $Y$, the given equation $A_{i} = Q_{s}R_{i}+Q_{i}R_{s}$ in (\ref{equ: thin As=QsRs Ai = QsRi + QiRs}) transforms into a Sylvester equation involving $Q_{i}$ and $R_{i}$. Utilizing  Theorem \ref{the: solution of Sylvester equation}, if and only if
	\begin{equation}\label{equ: (I-QsQsT)Ai(I-RsRs)}
		(I_{m}-Q_{s}Q_{s}^{\dagger})A_{i}(I_{n}-R_{s}^{\dagger}R_{s}) = O_{m\times n},
	\end{equation}
	which holds true identically by the invertibility nature of $R_{s}$ and the fact $R_{s}^{\dagger} = R_{s}^{-1}$, one can derive the solution
	\begin{equation}\label{equ: thin RiQi after Sylvester}
		\begin{cases}
			Q_{i} &= -(I_{m}-Q_{s}Q_{s}^{\dagger})A_{i}(-R_{s})^{\dagger}+Z-(I_{m}-Q_{s}Q_{s}^{\dagger})Z(-R_{s})(-R_{s})^{\dagger} \\
			R_{i} &= Q_{s}^{\dagger}A_{i}+Q_{s}^{\dagger}Z(-R_{s})+(I_{n}-Q_{s}^{\dagger}Q_{s})W
		\end{cases}
	\end{equation}
	where $Z \in \mathbb{R}^{m \times n}$ and $W \in \mathbb{R}^{n \times n}$ are arbitrary.
	
	Considering the fact $Q_{s}^{\dagger} = Q_{s}^{\top}$ and $R_{s}^{\dagger}=R_{s}^{-1}$, we have
	\begin{equation}\label{equ: Qi = (I-QQT)AiRs-1 + QQTZ}
		\begin{aligned}
			Q_{i} & = -(I_{m}-Q_{s}Q_{s}^{\dagger})A_{i}(-R_{s})^{\dagger}+Z-(I_{m}-Q_{s}Q_{s}^{\dagger})Z(-R_{s})(-R_{s})^{\dagger} \\
			& = (I_{m}-Q_{s}Q_{s}^{\top})A_{i}R_{s}^{-1}+Z-(I_{m}-Q_{s}Q_{s}^{\top})ZR_{s}R_{s}^{-1} \\
			& = (I_{m}-Q_{s}Q_{s}^{\top})A_{i}R_{s}^{-1} + Q_{s}Q_{s}^{\top}Z,
		\end{aligned}
	\end{equation}
	and
	\begin{equation}\label{equ: Ri =QTA-QsTZR}
		\begin{aligned}
			R_{i} &= Q_{s}^{\dagger}A_{i}+Q_{s}^{\dagger}Z(-R_{s})+(I_{n}-Q_{s}^{\dagger}Q_{s})W \\
			& = Q_{s}^{\top}A_{i}-Q_{s}^{\top}ZR_{s},
		\end{aligned}
	\end{equation}
	where $Z \in \mathbb{R}^{m \times n}$ is arbitrary.
	
	Since $Q_{s}\in \mathbb{R}^{m \times n}$ is a matrix with orthogonal columns, there exists a matrix $\widehat{Q}_{s} \in \mathbb{R}^{m \times (m-n)}$ with orthogonal columns, such that $[Q_{s},\widehat{Q}_{s}]$ is an orthogonal matrix. Thus, $Z = Q_{s}P + \widehat{Q}_{s}\widehat{P}$, where $P\in \mathbb{R}^{n \times n}$ and $\widehat{P} \in \mathbb{R}^{(m-n)\times n}$ are arbitrary matrices. At this moment, we have $Q_{s}^{\top}Z = Q_{s}^{\top}(Q_{s}P + \widehat{Q}_{s}\widehat{P}) = P$.
	Then the expressions for $Q_{i}$ and $R_{i}$ are as follows,
	\begin{equation}\label{equ: Qi Ri in thin proof}
		\begin{cases}
			Q_{i} & = (I_{m}-Q_{s}Q_{s}^{\top})A_{i}R_{s}^{-1} + Q_{s}P \\
			R_{i}& = Q_{s}^{\top}A_{i}-PR_{s},
		\end{cases}
	\end{equation}
	where $P\in \mathbb{R}^{n \times n}$ is an arbitrary matrix.
	
	Just as we discussed in  Theorem \ref{the: QR decomposition fo dual matrices}, we aim to exploit the orthogonal column properties of $Q$ and the upper triangular nature of $R$ to ascertain the values of $P$. In this regard, we first contemplate the scenario where $Q$ is a dual matrix with orthogonal columns. Thus, we have
	\begin{equation}\label{equ: QsQi+QiQs=On}
		Q_{s}^{\top}Q_{i} + Q_{i}^{\top}Q_{s} = O_{n }.
	\end{equation}
	By substituting the expression for $Q_{i}$ from (\ref{equ: Qi Ri in thin proof}), we can obtain
	\begin{equation}\label{equ: O=P+PT}
		\begin{aligned}
			O_{n} & = Q_{s}^{\top}[(I_{m}-Q_{s}Q_{s}^{\top})A_{i}R_{s}^{-1} + Q_{s}P] + [(I_{m}-Q_{s}Q_{s}^{\top})A_{i}R_{s}^{-1} + Q_{s}P]^{\top}Q_{s} \\
			& = Q_{s}^{\top}Q_{s}P +(Q_{s}P)^{\top}Q_{s} \\
			& = P+P^{\top}
		\end{aligned}
	\end{equation}
	with the fact $Q_{s}^{\top}(I_{m}-Q_{s}Q_{s}^{\top}) = O_{n}$.
	
	We deduce that $P\in \mathbb{R}^{n \times n}$ is a skew-symmetric matrix from (\ref{equ: O=P+PT}). Next, we leverage the upper triangular property of $R$ to ascertain the elements within matrix $P$. Considering the expression of $R_{i}$ in (\ref{equ: Qi Ri in thin proof}), we have
	\begin{equation}\label{equ: PRs = QsTAi-Ri thin}
		PR_{s} = Q_{s}^{\top}A_{i}-R_{i}.
	\end{equation}
	We denote $B = Q_{s}^{\top}A_{i} \in\mathbb{R}^{n \times n}$ and $B = [b_{1},b_{2},\ldots,b_{n}], R_{i} = [r_{i1},r_{i2},\ldots,r_{in}], R_{s} = [r_{s1},r_{s2},\ldots,r_{sn}]$ where $b_{k},r_{ik},r_{sk}\in \mathbb{R}^{n},k = 1,2,\ldots,n $. By partitioning the matrices column-wise and given that $R_s$ is of
	full rank, (\ref{equ: PRs = QsTAi-Ri thin}) is transformed into the following system of $n$ equations,
	\begin{equation}\label{equ: Prs=B1-Ri1 thin}
		\begin{cases}
			Pr_{s1} &= b_{1}-r_{i1}, \\
			Pr_{s2} &= b_{2}-r_{i2}, \\
			& \ldots \\
			Pr_{sn} &= b_{n}-r_{in}.\\
		\end{cases}
	\end{equation}
	Exploiting that matrix $R_{s}$ is upper triangular, along with  $R_{s} =(r_{s_{ij}})\in \mathbb{R}^{n\times n}$ and $P = [p_{1},p_{2},\ldots,p_{n}]$ where $p_{k}\in \mathbb{R}^{n},k = 1,2,\ldots n$, the system (\ref{equ: Prs=B1-Ri1 thin}) can be transmuted into
	\begin{equation}\label{equ: rs11p1=b1-ri1 thin}
		\begin{cases}
			r_{s_{11}}p_{1} &= b_{1}-r_{i1}, \\
			r_{s_{12}}p_{1}+r_{s_{22}}p_{2} &= b_{2}-r_{i2}, \\
			& \ldots \\
			r_{s_{1n}}p_{1}+r_{s_{2n}}p_{2}+\ldots+r_{s_{nn}}p_{n} &= b_{n}-r_{in}.
		\end{cases}
	\end{equation}
	Recognizing the upper triangular structure of $R_{i}$, it follows that for $r_{ik},k = 1,2,\ldots, (n-1)$, the lower $n-k$ entries $r_{ik}(k+1:n)$ are all zeros. Consequently, $p_{k}$ also requires computation solely for the lower $n-k$ entries $p_{k}(k+1:n)$, thereby ensuring the maintenance of $P$'s skew-symmetric property. Hence, the system of $(n-1)$ equations has the solutions
	\begin{equation}
		\begin{cases}
			p_{1}(2:n) &= b_{1}(2:n)/r_{s_{11}}, \\
			p_{2}(3:n) &= (b_{2}(3:n)-r_{s_{12}}p_{1}(3:n))/r_{s_{22}}, \\
			& \ldots \\
			p_{n-1}(n) &= (b_{n-1}(n)-\sum\limits_{t = 1}^{n-2}r_{s_{t(n-1)}}p_{t}(n))/r_{s_{(n-1)(n-1)}}.
		\end{cases}
	\end{equation}
	Thus, we complete the proof.
\end{proof}
\bigskip

The proof of Theorem \ref{the: RDQRCP}.\\

\begin{proof}
	Suppose randomized QR decomposition with column pivoting in the standard part of $A$ exists, then we have that
	\begin{equation}\label{equ: rand As=QsRs Ai = QsRi + QiRs}
		\begin{cases}
			A_{s}P_{s} = Q_{s}R_{s}\\
			A_{i}P_{s} = Q_{s}R_{i}+Q_{i}R_{s}
		\end{cases}
	\end{equation}
	through computation directly. By treating $R_{i}$ as $X$ and $Q_{i}$ as $Y$, the given equation $A_{i}P_{s} = Q_{s}R_{i}+Q_{i}R_{s}$ in (\ref{equ: rand As=QsRs Ai = QsRi + QiRs}) transforms into a Sylvester equation involving $Q_{i}$ and $R_{i}$. Utilizing Theorem \ref{the: solution of Sylvester equation} and the fact $Q_{s}^{\dagger}=Q_{s}^{\top}$, if and only if
	\begin{equation}\label{equ: (I-QsQsT)AiP_{s}(I-RsRs)}
		(I_{m}-Q_{s}Q_{s}^{\top})A_{i}P_{s}(I_{n}-R_{s}^{\dagger}R_{s}) = O_{m\times n}
	\end{equation}
	holds true, one can derive the solution
	\begin{equation}\label{equ: rand RiQi after Sylvester}
		\begin{cases}
			Q_{i} &= -(I_{m}-Q_{s}Q_{s}^{\dagger})A_{i}P_{s}(-R_{s})^{\dagger}+Z-(I_{m}-Q_{s}Q_{s}^{\dagger})Z(-R_{s})(-R_{s})^{\dagger} \\
			R_{i} &= Q_{s}^{\dagger}A_{i}P_{s}+Q_{s}^{\dagger}Z(-R_{s})+(I_{k}-Q_{s}^{\dagger}Q_{s})W
		\end{cases}
	\end{equation}
	where $Z \in \mathbb{R}^{m \times k}$ and $W \in \mathbb{R}^{k \times n}$ are
	arbitrary.
	
	Considering the fact $Q_{s}^{\dagger} = Q_{s}^{\top}$ and $R_{s}R_{s}^{\dagger} = I_{k}$, we have
	\begin{equation}\label{equ: Qi = (I-QsQs)APR}
		\begin{aligned}
			Q_{i} & = -(I_{m}-Q_{s}Q_{s}^{\dagger})A_{i}P_{s}(-R_{s})^{\dagger}+Z-(I_{m}-Q_{s}Q_{s}^{\dagger})Z(-R_{s})(-R_{s})^{\dagger} \\
			& = (I_{m}-Q_{s}Q_{s}^{\top})A_{i}P_{s}R_{s}^{\dagger}+Z-(I_{m}-Q_{s}Q_{s}^{\top})ZR_{s}R_{s}^{\dagger} \\
			& = (I_{m}-Q_{s}Q_{s}^{\top})A_{i}P_{s}R_{s}^{\dagger} + Q_{s}Q_{s}^{\top}Z,
		\end{aligned}
	\end{equation}
	and
	\begin{equation}\label{equ: Ri =QTAP-QsTZR}
		\begin{aligned}
			R_{i} &= Q_{s}^{\dagger}A_{i}P_{s}+Q_{s}^{\dagger}Z(-R_{s})+(I_{n}-Q_{s}^{\dagger}Q_{s})W \\
			& = Q_{s}^{\top}A_{i}P_{s}-Q_{s}^{\top}ZR_{s},
		\end{aligned}
	\end{equation}
	where $Z \in \mathbb{R}^{m \times k}$ is arbitrary.
	
	Since $Q_{s}\in \mathbb{R}^{m \times k}$ is a matrix with orthogonal columns, there exists a matrix $\widehat{Q}_{s} \in \mathbb{R}^{m \times (m-k)}$ with orthogonal columns, such that $[Q_{s},\widehat{Q}_{s}]$ is an orthogonal matrix. Thus, $Z = Q_{s}P + \widehat{Q}_{s}\widehat{P}$, where $P\in \mathbb{R}^{k \times k}$ and $\widehat{P} \in \mathbb{R}^{(m-k)\times k}$ are arbitrary matrices. At this moment, we have $Q_{s}^{\top}Z = Q_{s}^{\top}(Q_{s}P + \widehat{Q}_{s}\widehat{P}) = P$ .Then the expressions for $Q_{i}$ and $R_{i}$ are as follows,
	\begin{equation}\label{equ: Qi Ri in rand proof}
		\begin{cases}
			Q_{i} & = (I_{m}-Q_{s}Q_{s}^{\top})A_{i}P_{s}R_{s}^{\dagger} + Q_{s}P \\
			R_{i}& = Q_{s}^{\top}A_{i}P_{s}-PR_{s},
		\end{cases}
	\end{equation}
	where $P\in \mathbb{R}^{k \times k}$ is an arbitrary matrix.
	
	First, we contemplate the scenario where $Q$ is a dual matrix with orthogonal columns. Thus, we have
	\begin{equation}\label{equ: QsQi+QiQs=Ok}
		Q_{s}^{\top}Q_{i} + Q_{i}^{\top}Q_{s} = O_{k }.
	\end{equation}
	By substituting the expression for $Q_{i}$ from (\ref{equ: Qi Ri in rand proof}), we can obtain
	\begin{equation}\label{equ: Ok=P+PT}
		\begin{aligned}
			O_{k } & = Q_{s}^{\top}[(I_{m}-Q_{s}Q_{s}^{\top})A_{i}P_{s}R_{s}^{\dagger} + Q_{s}P] + [(I_{m}-Q_{s}Q_{s}^{\dagger})A_{i}P_{s}R_{s}^{\dagger} + Q_{s}P]^{\top}Q_{s} \\
			& = Q_{s}^{\top}Q_{s}P +(Q_{s}P)^{\top}Q_{s} \\
			& = P+P^{\top}
		\end{aligned}
	\end{equation}
	with the fact $Q_{s}^{\top}(I_{m}-Q_{s}Q_{s}^{\top}) = O_{k}$.
	
	We deduce that $P\in \mathbb{R}^{k \times k}$ is a skew-symmetric matrix from (\ref{equ: Ok=P+PT}). Next, we leverage the upper trapezoidal property of $R$ to ascertain the elements within matrix $P$. Considering the expression of $R_{i}$ in (\ref{equ: Qi Ri in rand proof}), we have
	\begin{equation}\label{equ: PRs = QsTAi-Ri rand}
		PR_{s} = Q_{s}^{\top}A_{i}P_{s}-R_{i}.
	\end{equation}
	We denote $B = Q_{s}^{\top}A_{i}P_{s} \in\mathbb{R}^{k \times n}$ and $B = [b_{1},b_{2},\ldots,b_{n}], R_{i} = [r_{i1},r_{i2},\ldots,r_{in}], R_{s} = [r_{s1},r_{s2},\ldots,r_{sn}]$ where $b_{t},r_{it},r_{st}\in \mathbb{R}^{k},t = 1,2,\ldots,n $. By partitioning the matrices column-wise and given that $R_s$ is of full rank, (\ref{equ: PRs = QsTAi-Ri rand}) is transformed into the following system of $n$ equations,
	\begin{equation}\label{equ: Prs=B1-Ri1 rand}
		\begin{cases}
			Pr_{s1} &= b_{1}-r_{i1}, \\
			Pr_{s2} &= b_{2}-r_{i2}, \\
			& \ldots \\
			Pr_{sn} &= b_{n}-r_{in}.\\
		\end{cases}
	\end{equation}
	Exploiting that matrix $R_{s}$ is upper trapezoidal, along with  $R_{s} = (r_{s_{ij}})\in \mathbb{R}^{k\times n}$ and $P = [p_{1},p_{2},\ldots,p_{k}]$ where $p_{t}\in \mathbb{R}^{k},t = 1,2,\ldots n$, the overdetermined system (\ref{equ: Prs=B1-Ri1 rand}) can be transmuted into
	\begin{equation}\label{equ: rs11p1=b1-ri1 rand}
		\begin{cases}
			r_{s_{11}}p_{1} &= b_{1}-r_{i1}, \\
			r_{s_{12}}p_{1}+r_{s_{22}}p_{2} &= b_{2}-r_{i2}, \\
			& \ldots \\
			r_{s_{1n}}p_{1}+r_{s_{2n}}p_{2}+\ldots+r_{s_{kn}}p_{k} &= b_{n}-r_{in}.
		\end{cases}
	\end{equation}
	Recognizing the upper trapezoidal structure of $R_{i}$, it follows that for $r_{it},t = 1,2,\ldots, (k-1)$, the lower $k-t$ entries $r_{it}(t+1:k)$ are all zeros. Consequently, $p_{t}$ also requires computation solely for the lower $k-t$ entries $p_{t}(t+1:k)$, thereby ensuring the maintenance of $P$'s skew-symmetric property. Moreover, the consistency of an over-determined system (\ref{equ: Prs=B1-Ri1 rand}) can be uniquely determined by $R_{i}(:,k+1:n)$. Hence, the over-determined system of $(k-1)$ equations has the solutions
	\begin{equation}
		\begin{cases}
			p_{1}(2:k) &= b_{1}(2:k)/r_{s_{11}}, \\
			p_{2}(3:k) &= (b_{2}(3:k)-r_{s_{12}}p_{1}(3:k))/r_{s_{22}}, \\
			& \ldots \\
			p_{k-1}(k) &= (b_{k-1}(k)-\sum\limits_{t = 1}^{k-2}r_{s_{t(k-1)}}p_{t}(k))/r_{s_{(k-1)(k -1)}}.
		\end{cases}
	\end{equation}
	Thus, we complete the proof.
\end{proof}
}
\section*{Data Availability}
The numerical data shown in the figures are available upon request, and so is the code to	run the experiment and generate the figures.

\section*{Declarations}
Conflicts of Interest.  The authors declare no competing interests.

\footnotesize
\bibliographystyle{abbrv}
\bibliography{reference}

\end{document}